\documentclass[12pt,reqno]{amsart}

\usepackage[utf8]{inputenc}
\usepackage{amssymb}
\usepackage{amsmath}
\usepackage{amsthm}
\usepackage{amsfonts}
\usepackage{bbm}
\usepackage{bookmark}
\usepackage{hyperref}
\hypersetup{hidelinks, pdfstartview={FitH}}
\usepackage{color}
\usepackage{xcolor}
\usepackage{mathrsfs}
\usepackage[normalem]{ulem}
\usepackage{caption}
\usepackage{graphicx}
\usepackage{tikz}

\theoremstyle{plain}
\newtheorem{theorem}{Theorem}[section]
\newtheorem{thma}{Theorem}
\newtheorem{lemma}[theorem]{Lemma}
\newtheorem{prop}[theorem]{Proposition}
\newtheorem{corollary}[theorem]{Corollary}

\theoremstyle{definition}

\theoremstyle{remark}
\newtheorem{rem}[theorem]{Remark}

\numberwithin{equation}{section}

\newcommand{\g}{{\mathbbm g}}
\newcommand{\h}{{\mathbbm h}}

\DeclareMathOperator{\supp}{supp}

\DeclareMathOperator{\pv}{p.v.}

\DeclareFontFamily{U}{mathx}{\hyphenchar\font45}
\DeclareFontShape{U}{mathx}{m}{n}{
<5> <6> <7> <8> <9> <10>
<10.95> <12> <14.4> <17.28> <20.74> <24.88>
mathx10
}{}
\DeclareSymbolFont{mathx}{U}{mathx}{m}{n}
\DeclareFontSubstitution{U}{mathx}{m}{n}
\DeclareMathAccent{\widecheck}{0}{mathx}{"71}

\theoremstyle{plain}

\title[3D curved simplex Hilbert transform]{A curved three-dimensional simplex Hilbert transform}

\author[Durcik]{Polona Durcik}
\address{Schmid College of Science and Technology, Chapman University, Orange, CA, USA}
\email{durcik@chapman.edu}

\author[Fraccaroli]{Marco Fraccaroli}
\address{Department of Mathematics and Statistics, University of Massachusetts Lowell, Lowell, MA, USA}
\email{marco\_fraccaroli@uml.edu}

\author[Roos]{Joris Roos}
\address{Department of Mathematics and Statistics, University of Massachusetts Lowell, Lowell, MA, USA}
\email{joris\_roos@uml.edu}

\subjclass[2020]{42B15, 42B20, 42B25}

\keywords{multilinear singular integrals, multilinear oscillatory integrals, simplex Hilbert transform}

\date{\today}

\begin{document}

\begin{abstract}
We prove Lebesgue space estimates for a curved three-dimensional simplex Hilbert transform associated with the moment curve $t \mapsto (t,t^2,t^3)$. The analysis relies on smoothing inequalities for curved multilinear averages and on estimates for a three-dimensional anisotropic variant of the twisted paraproduct. The result has been formally verified in Lean 4.
\end{abstract}

\maketitle

\section{Introduction}
Let $N$ be a positive natural number. Consider the $N$-linear operator
\begin{equation}\label{eqn:curvedsimplex}
\mathcal{T}_N (\mathbf{f})(x) = \pv \int_{{{\mathbb R}}} \prod_{n=1}^N f_n(x + t^n e_n) \frac{dt}{t},
\end{equation}
initially defined for $N$-tuples $\mathbf{f} = (f_1,\dots,f_N)$ of Schwartz functions on ${{\mathbb R}}^N$.
This can be viewed as an $N$-dimensional simplex Hilbert transform associated with the moment curve $t \mapsto (t,t^2,\dots,t^N)$.
The operator $\mathcal{T}_1$ coincides with the classical Hilbert transform, and the two-dimensional operator $\mathcal{T}_2$ is also known to satisfy Lebesgue space estimates \cite{CDR21}.
This paper focuses on the three-dimensional case $N=3$.
\begin{thma} \label{thm:main}
Let $p \in (1, \tfrac{4}{3} )$, $p_1,p_2,p_3\in (1,4)$ such that
$\frac{1}{p_1} + \frac{1}{p_2} + \frac{1}{p_3} = \frac{1}{p}.$
Then there exists a constant $C\in (0,\infty)$ such that, for every triple $\mathbf{f}$ of Schwartz functions,
\begin{equation*}
\|\mathcal{T}_3(\mathbf{f})\|_{p} \le C\prod_{n=1}^3 \|f_n\|_{p_n}.
\end{equation*}
\end{thma}

To put this result into context we describe related problems and results. Consider the \emph{simplex Hilbert transform}
\begin{equation*}
H_N (\mathbf{f})(x) = \pv \int_{{{\mathbb R}}} \prod_{n=1}^N f_n(x+t e_n) \frac{dt}{t},
\end{equation*}
initially defined for $N$-tuples $\mathbf{f} = (f_1,\dots,f_N)$ of Schwartz functions on ${{\mathbb R}}^N$, where $e_1, \dots, e_N$ are the standard unit vectors in ${{\mathbb R}}^N$. Therefore $H_1$ is the classical Hilbert transform and $H_2$ is the \emph{triangular Hilbert transform}. Establishing Lebesgue space estimates for the simplex Hilbert transform $H_N$ when $N \geq 2$ is a major open problem in harmonic analysis. Kova\v{c}, Thiele, and Zorin-Kranich \cite{KTZ15} obtained partial results for a Walsh model in the case $N=2$, under additional assumption on the shape of one of the functions $f_n$. For truncated simplex Hilbert transforms, Zorin-Kranich \cite{ZK17} proved non-trivial cancellation results in terms of the number of scales considered. This was further improved in \cite{DKT19}.

The difficulty of the problem for the simplex Hilbert transform $H_N$ is also reflected in its relation to the \emph{$N$-linear Hilbert transform}
\begin{equation*}
L_N (\mathbf{f})(x) = \pv \int_{{{\mathbb R}}} \prod_{n=1}^N f_n(x+ n t) \frac{dt}{t}
\end{equation*}
initially defined for $N$-tuples $\mathbf{f} = (f_1,\dots,f_N)$ of Schwartz functions on ${{\mathbb R}}$. Lebesgue space estimates for $H_N$ imply estimates for $L_N$ in the same range of exponents. The operator $L_2$ is the \emph{bilinear Hilbert transform}, for which Lebesgue space estimates were shown by Lacey and Thiele \cite{LT97, LT99}. For $N \geq 3$, establishing such estimates for $L_N$ is a significant open problem
\footnote{On October 6, 2026 around 6 pm EDT,
OpenAI made public several preprints claiming, among other things, that the triangular Hilbert transform $H_2$ is bounded on $L^3\times L^3\to L^{3/2}$ and the trilinear Hilbert transform $L_3$ is bounded on $L^3\times L^3\times L^3\to L^1$.}.

This paper continues the research program on curved variants of these operators. The presence of curvature changes the nature of these operators and introduces oscillatory cancellation that can be quantified in terms of smoothing inequalities for corresponding averaging operators. The curved variant of $L_N$ is
\begin{equation*}
S_N (\mathbf{f})(x) = \pv \int_{{{\mathbb R}}} \prod_{n=1}^N f_n(x + t^n ) \frac{dt}{t},
\end{equation*}
initially defined for $N$-tuples $\mathbf{f} = (f_1,\dots,f_N)$ of Schwartz functions on ${{\mathbb R}}$.

As for $H_N$ and $L_N$, every Lebesgue space estimate for $\mathcal{T}_N$ implies the corresponding estimate for $S_N$. For $N=2$, the reduction is outlined in \cite[\S 5.2]{CDR21} adapting the argument in \cite[\S B.3]{KTZ15}, but it applies in arbitrary dimension $N \in {{\mathbb N}}$.

However, the simplex operators $H_N$ and $\mathcal{T}_N$ satisfy additional modulation symmetries due to their coordinatewise structure. For $N \geq 2$, each function $f_n$ in the definition of $H_N$ and $\mathcal{T}_N$ is translated only in one coordinate direction. Therefore, an arbitrary phase multiplying $f_n$ and independent of $x_n$ factors out of the integral. This symmetry is not present in the one-dimensional operators $L_N$ and $S_N$.

A foundational example of singular integral operators exhibiting such general modulation symmetries is the twisted paraproduct, which appears as a degenerate case of the two-dimensional bilinear Hilbert transform studied by Demeter and Thiele \cite{DT10}. Kova\v{c} \cite{Kov12} first proved Lebesgue space estimates for the twisted paraproduct.

For the curved one-dimensional operator, Li \cite{Li13} proved the estimate $L^2 \times L^2 \to L^1$ for $S_2$. Lie \cite{Lie15} gave a different proof for a broader class of curves. See also the results of Li and Xiao \cite{LX16} and Lie \cite{Lie18} for further Lebesgue space estimate. Hu and Lie \cite{HL23, HL25} showed Lebesgue space estimates for $S_3$, and Guo, Li, and Zhan \cite{GLZ26} proved Lebesgue space estimates for $S_N$, for arbitrary $N \geq 3$. The latter paper relies on the smoothing inequality of Kosz, Mirek, Peluse, Wan, and Wright \cite{KMPWW24}; also see \cite{BK25} for an application of that inequality to maximal function analogues. The smoothing inequality from \cite{KMPWW24} is also a main ingredient of our proof.

Lebesgue space estimates for $\mathcal{T}_2$ were first proved in \cite{CDR21}, also see \cite{HLS24} for an independent exposition. A main ingredient was a smoothing inequality for the associated averages. See \cite{HL26} for a generalization and alternative proof of this smoothing inequality, and see \cite{Chr25,Chr22,CZ24,DR24,Bou88} for more on closely related inequalities.

Our argument here follows the same general strategy as in \cite{CDR21}. The argument has three main components:

{\em 1. Reduction (\S \ref{sec:decomposition}).} A stationary phase analysis naturally splits the operator into dyadic frequency annuli in each variable. We only expect troublesome contributions from pieces involving stationary behavior.
In the two-dimensional setting of \cite{CDR21}, a simple dichotomy arises: either both frequencies are large (high-frequency case), or both of them are small (low-frequency case). In the present three-dimensional setting there are additional mixed-frequency cases that can also exhibit stationary behavior, which complicate the analysis somewhat.

{\em 2. Smoothing inequalities (\S \ref{sec:smoothing}).} The high-frequency contributions of the operator are estimated using smoothing inequalities providing appropriate decay. In this paper we use both, a three-dimensional and a two-dimensional smoothing inequality, see Theorem \ref{thm:smoothing} and Theorem \ref{thm:smoothing_2}.
The three-dimensional inequality is deduced from a special case of \cite[Thm. 6.1]{KMPWW24}. The two-dimensional inequality is a variant of \cite[Theorem 5]{CDR21} and is used to deal with certain mixed-frequency cases.

{\em 3. Anisotropic twisted paraproducts (\S \ref{sec:twisted}).} The third component deals with the low-frequency contributions to the operator. In essence, the low-frequency cutoffs smooth out the highly singular one-dimensional Hilbert kernel to a full-dimensional standard anisotropic Calder\'on--Zygmund kernel. This leads to multilinear oeprators
\begin{equation}\label{eqn:twisted}
\Gamma_N(\mathbf{f})(x) = \int_{{{\mathbb R}}^3} K(t) \prod_{n=1}^N f_n(x + t_n e_n)\,dt,
\end{equation}
where $K$ is an anisotropic Calder\'on--Zygmund kernel on ${{\mathbb R}}^N$.
When $K$ is isotropic and $N=2$, then $\Gamma_2$ is Kova\v{c}'s original twisted paraproduct \cite{Kov12}. When $K$ is anisotropic and $N=2$, then $\Gamma_2$ is the operator appearing in \cite{CDR21}.
Theorem \ref{thm:main_twist} proves the required Lebesgue space bounds for $\Gamma_3$ following arguments from \cite{CDR21}, also see \cite{Kov12}, \cite{Ber12}, \cite[\S 3]{DR21}.

In the higher-dimensional cases $N\ge 4$, the main obstruction lies in this third component: no bounds are currently known for the twisted paraproduct $\Gamma_N$ when $N\ge 4$.

\subsection{Smoothing inequalities}
For a compactly supported function $\psi \in \mathcal{C}^1({{\mathbb R}})$, we consider the averaging operator
\begin{equation*}
\mathcal{A}_\psi (\mathbf{f}) (x) = \int_{\mathbb{R}} \psi(t) \prod_{n=1}^3 f_n(x+t^n e_n) dt ,
\end{equation*}
initially defined for triples $\mathbf{f} = (f_1,f_2,f_3)$ of Schwartz functions on ${{\mathbb R}}^3$.
We have the following estimate.
\begin{thma} \label{thm:smoothing}
Let $\psi\in \mathcal{C}^1(\mathbb{R})$ with $\mathrm{supp} (\psi) \subset [a,b]$, where $a,b \in (0, \infty)$.
Let $p \in [1,\infty )$, $p_1,p_2,p_3\in (1,\infty)$ such that
\begin{equation*}
\frac{1}{p_1} + \frac{1}{p_2} + \frac{1}{p_3} = \frac{1}{p}.
\end{equation*}
Then there exists a constant $\delta = \delta(p, p_1,p_2,p_3) >0$ such that the following holds.

Let $\lambda\ge 1$ and suppose $\supp (\widehat{f_n}) \subseteq \{|\xi_n|\ge \lambda\}$ for at least one $n \in \{1,2,3\}$. Then, for every triple $\mathbf{f}$ of Schwartz functions,
\begin{equation}\label{eq:smoothing}
\|\mathcal{A}_{\psi}( \mathbf{f} )\|_{p}
\le C \|\psi'\|_1 \lambda^{-\delta}\prod_{n=1}^3 \|f_n\|_{p_n},
\end{equation}
where $C\in (0,\infty)$ depends only on $a,b,p,p_1,p_2,p_3$.
\end{thma}
We derive this estimate from the general smoothing estimates of Kosz, Mirek, Peluse, Wan and Wright \cite[Thm. 6.1]{KMPWW24}.
We also need a version of this result for a bilinear averaging operator along curves defined by arbitrary monomials, with an explicit dependence of the multiplicative constant on the cutoff function. For our argument, only the $L^2 \times L^2 \to L^1$ estimate is required.

Let $\alpha_1, \alpha_2$ be positive distinct integers. For a compactly supported function $\psi \in \mathcal{C}^1({{\mathbb R}})$, we consider the averaging operator
\begin{equation*}
\mathbf{A}_\psi (\mathbf{g}) (x) = \int_{\mathbb{R}} \psi(t) \prod_{n=1}^2 g_n(x+t^{\alpha_n} e_n) dt ,
\end{equation*}
intially defined for pairs $\mathbf{g} = (g_1,g_2)$ of Schwartz functions on ${{\mathbb R}}^2$.
\begin{thma} \label{thm:smoothing_2}
Let $\psi\in \mathcal{C}^1(\mathbb{R})$ with $\mathrm{supp} (\psi) \subset [a,b]$, where $a,b \in (0, \infty)$.
Then there exists an absolute constant $\delta>0$ such that the following holds.

Let $\lambda\ge 1$ and suppose $\supp (\widehat{g_n}) \subseteq \{|\xi_n|\ge \lambda\}$ for at least one $n \in \{1,2\}$. Then for every pair $\mathbf{g}$ of Schwartz functions,
\begin{equation*}
\|\mathbf{A}_{\psi}( \mathbf{g} )\|_{1}
\le C \| \psi' \|_{1} \lambda^{-\delta} \|g_1\|_{2} \|g_2\|_{2},
\end{equation*}
where $C\in (0,\infty)$ is a constant depending only on $a,b,\alpha_1,\alpha_2$.
\end{thma}

\subsection{Anisotropic twisted paraproduct}
The second ingredient is a Lebesgue space estimate for an anisotropic twisted paraproduct.

Let $\alpha=(\alpha_1,\alpha_2,\alpha_3)$ be a triple of positive integers. Let $K \in \mathcal{S}'({{\mathbb R}}^3)$ be a tempered distribution such that $\widehat{K}$ is a smooth function on ${{\mathbb R}}^3\setminus\{0\}$ satisfying
\begin{equation} \label{eq:symbol_condition}
|\partial_\xi^k\widehat{K}(\xi)| \le \Big( \sum_{n=1}^3 |\xi_n|^{1/\alpha_n} \Big)^{-\alpha\cdot k}
\end{equation}
for all multi-indices $k=(k_1,k_2,k_3)$ up to a large finite order, say $|k|\le 110$. Recall the definition of $\Gamma_3$ from \eqref{eqn:twisted}.

\begin{thma} \label{thm:main_twist}
Let $p \in (1, \tfrac{4}{3} )$, $p_1,p_2,p_3\in (1,4)$ such that
\begin{equation*}
\frac{1}{p_1} + \frac{1}{p_2} + \frac{1}{p_3} = \frac{1}{p}.
\end{equation*}
Then there exists a constant $C\in(0,\infty)$ such that for every triple $\mathbf{f}$ of Schwartz functions,
\begin{equation*}
\|\Gamma_3(\mathbf{f})\|_{p} \le C\prod_{n=1}^3 \|f_n\|_{p_n}.
\end{equation*}
\end{thma}

\subsection*{Notation}
The Fourier transform of a function $f$ will be denoted by
\begin{equation*}
\widehat{f}(\xi) = \int_{{{\mathbb R}}^d} f(x) e^{-2\pi i x\cdot \xi} dx.
\end{equation*}

The letters $\delta$ and $C$ are reserved for constants that may change from line to line. Dependence of the constants on various parameters will be understood from context or made explicit appropriately.

We write $A\lesssim B$ to denote existence of a constant $C$ such that $A\le C B$.

For a function $f$ on ${{\mathbb R}}^d$, a function $\varphi$ on ${{\mathbb R}}$, and $n \in \{1, \dots, d\}$, we define the partial convolution
\begin{equation*}
(f*_n \varphi) (x) = \int_{{\mathbb R}} f(x- t e_n) \varphi(t) dt.
\end{equation*}

$L^p$ norms will be denoted by $\|f\|_p$, $\|f\|_{L^p}$, or $\|f\|_{L^p({{\mathbb R}}^d)}$, depending on context.

\subsection*{Statement on use of AI}
The results in this paper have been obtained by the authors, without the use of AI. All the mathematics and all of the writing in this paper are due to the authors.

AI coding agents were used to formally verify the result in the proof assistant Lean 4 \cite{LeanPaper} and using its mathematical library Mathlib \cite{Mathlib}. The formalization is available on \href{https://github.com/roos-j/lean-dfr}{GitHub}.

\subsection*{Acknowledgements}
This work was supported in part by NSF grants DMS-2554859 (P.D.), DMS-2555784 (J.R.), and grants from the Simons Foundation. M.F. was supported in part by an AMS-Simons Travel Grant.

\section{Reduction} \label{sec:decomposition}
In this section we follow the decomposition of the operator performed in \cite[Section 2]{CDR21}. In this way, we reduce the proof of estimates for $\mathcal{T}_3$, Theorem \ref{thm:main}, to smoothing inequalities for averages along curves, Theorem \ref{thm:smoothing} and Theorem \ref{thm:smoothing_2}, and the boundedness of an anisotropic twisted paraproduct, Theorem \ref{thm:main_twist}.

\subsection{Littlewood--Paley decompositions}\label{sec:stationary}
Let $\varphi$ be a smooth even function on ${{\mathbb R}}$ such that
\begin{equation*}
\mathbf{1}_{\{|\zeta|\le 1\}} (\zeta) \leq \varphi (\zeta) \leq \mathbf{1}_{\{|\zeta|\le 2\}} (\zeta).
\end{equation*}
Let $\psi$ be the function on ${{\mathbb R}}$ defined by
\begin{equation*}
\psi(\zeta) = \varphi(\zeta) - \varphi(2 \zeta),
\end{equation*}
and, for every $j\in {{\mathbb Z}}$, let $\varphi_j$ and $\psi_j$ be the functions on ${{\mathbb R}}$ defined by
\begin{equation*}
\varphi_j(\zeta) = \varphi(2^{-j} \zeta), \qquad \psi_j(\zeta) = \psi(2^{-j} \zeta) = \varphi_j(\zeta) - \varphi_{j-1}(\zeta).
\end{equation*}
Then $\psi_j$ is supported in $\{2^{j-1}\le |\zeta|\le 2^{j+1}\}$ and, for every $\zeta \neq 0$,
\begin{equation*}
\sum_{j\in{{\mathbb Z}}} \psi_j(\zeta) = 1 .
\end{equation*}
For all $\ell \in \{1,2,3\}$ and $j \in {{\mathbb Z}}$, we define the partial Littlewood--Paley operators $\Delta^{(\ell)}_j$ and corresponding partial sums $S_j^{(\ell)}$ by
\begin{equation*}
\widehat{\Delta^{(\ell)}_j f} (\xi) = \psi_j(\xi_\ell) \widehat{f}(\xi),\qquad \widehat{S^{(\ell)}_j f} (\xi) = \varphi_j(\xi_\ell) \widehat{f}(\xi),
\end{equation*}
where $\xi=(\xi_1,\xi_2,\xi_3)\in{{\mathbb R}}^3$.

Next, let $\widetilde{\psi}$ be a smooth even function on ${{\mathbb R}}$ such that
\begin{equation*}
\mathbf{1}_{\{\tfrac{1}{2} \le |\zeta|\le 1\}} (\zeta) \leq \widetilde{\psi} (\zeta) \leq \mathbf{1}_{\{\tfrac{1}{2} - \tfrac{1}{100} \le |\zeta|\le 1 + \tfrac{1}{100} \}} (\zeta),
\end{equation*}
and, for every $j\in {{\mathbb Z}}$, let $\widetilde{\psi}_j$ be the function on ${{\mathbb R}}$ be defined by
\begin{equation*}
\widetilde{\psi}_j(\zeta) = \widetilde{\psi}(2^{-j} \zeta).
\end{equation*}
In particular, $\widetilde{\psi}_j \psi_j = \psi_j$, hence $\Delta_j \widetilde{\Delta}_j = \widetilde{\Delta}_j \Delta_j = \Delta_j$.

We decompose the operator $\mathcal{T}_3$ in multiple ways. First, the dilation invariance motivates the decomposition by localizing the kernel at different scales, namely $\mathcal{T}_3=\sum_{j\in{{\mathbb Z}}} T_j$, where
\begin{equation*}
T_j(\mathbf{f})(x) = \int_{{{\mathbb R}}} \psi(2^{j} t) \prod_{n=1}^3 f_n(x + t^n e_n) \frac{dt}{t}.
\end{equation*}

Next, standard stationary phase considerations motivate the decomposition by localizing the functions in frequency at different scales, namely
\begin{equation}\label{eqn:basicfreqdecomp}
T = T^{\mathrm{L}} + T^{\mathrm{ML}} + T^{\mathrm{MH}} + T^{\mathrm{H}},
\end{equation}
where the four components represent the {\em low} (L), {\em mixed} (ML and MH) and {\em high} (H) frequency contributions, respectively.
In details, for every $k=(k_1,k_2,k_3)\in {{\mathbb Z}}^3$, let $k^{(1)}, k^{(2)}, k^{(3)}$ denote the components sorted in descending order, so that $k^{(1)}\ge k^{(2)} \ge k^{(3)}$. Then ${{\mathbb Z}}^3=\mathfrak{F}_\mathrm{L} \cup \mathfrak{F}_\mathrm{ML} \cup \mathfrak{F}_\mathrm{MH} \cup \mathfrak{F}_\mathrm{H}$ where
\begin{align*}
& \mathfrak{F}_\mathrm{L} = \{ k\in{{\mathbb Z}}^3 \colon k^{(1)}\le 0 \}, \\
& \mathfrak{F}_\mathrm{ML} = \{ k\in{{\mathbb Z}}^3 \colon k^{(1)}> 0, k^{(1)}-k^{(2)} > 100 \},
\\
& \mathfrak{F}_\mathrm{MH} = \{ k\in{{\mathbb Z}}^3 \colon k^{(1)}> 0, k^{(1)}-k^{(2)} \leq 100, k^{(1)}-k^{(3)} > 100 \}, \\
& \mathfrak{F}_\mathrm{H} = \{ k\in{{\mathbb Z}}^3 \colon k^{(1)}> 0, k^{(1)}-k^{(3)} \le 100 \}.
\end{align*}
For every $\omega\in\{\mathrm{L},\mathrm{ML},\mathrm{MH},\mathrm{H}\}$, we write
\begin{equation*}
T^\omega = \sum_{j\in{{\mathbb Z}}} T^\omega_j,\qquad T_j^\omega (\mathbf{f}) = \sum_{k \in \mathfrak{F}_\omega} T_j(\Delta^{(1)}_{j+k_1} f_1, \Delta^{(2)}_{2j+k_2} f_2, \Delta^{(3)}_{3j+k_3} f_3).
\end{equation*}
We split further $\mathfrak{F}_\mathrm{ML} =\mathfrak{F}_{\mathrm{ML},1}\cup \mathfrak{F}_{\mathrm{ML},2}\cup\mathfrak{F}_{\mathrm{ML},3}$,
where, for every $n=1,2,3$, we define
\begin{equation*}
\mathfrak{F}_{\mathrm{ML},n} = \{k\in \mathfrak{F}_{\mathrm{ML}} \colon k^{(1)} = k_n \}.
\end{equation*}
We also split further $\mathfrak{F}_\mathrm{MH} = \mathfrak{F}_\mathrm{MH,1} \cup \mathfrak{F}_\mathrm{MH,2} \cup \mathfrak{F}_\mathrm{MH,3}$,
where, for every $n=1,2,3$, we define
\begin{equation*}
\mathfrak{F}_{\mathrm{MH},n} = \{k\in \mathfrak{F}_{\mathrm{MH}} \colon k^{(3)} = k_n \}.
\end{equation*}

Finally, we decompose $T = \sum_{k \in {{\mathbb Z}}^3} T^{(k)}$, where
\begin{equation} \label{eq:decomposition_k}
T^{(k)}(\mathbf{f}) = \sum_{j \in {{\mathbb Z}}} T_j(\Delta^{(1)}_{j+k_1} f_1, \Delta^{(2)}_{2j+k_2} f_2, \Delta^{(3)}_{3j+k_3} f_3).
\end{equation}

The analysis of the low frequency $T^{\mathrm{L}}$, mixed frequency $T^{\mathrm{ML}}$, and high frequency $T^{\mathrm{H}}$ components is analogous to that carried out in \cite{CDR21}. Estimates for these components are proved in \S \ref{subsec:low}, \S \ref{subsec:mixedlow}, and \S \ref{subsec:high} respectively. We briefly recall that $T^{\mathrm{H}}$ relies on the smoothing effect of multilinear averages along curves for functions with frequency support away from the origin.
Instead, $T^{\mathrm{L}}$ is treated as a three-dimensional anisotropic variant of the twisted paraproduct \cite{Kov12}. Furthermore, relying on the rapid decay from non-stationary phase, $T^{\mathrm{ML}}$ is also reduced to a similar twisted paraproduct. It is worth noting that in the last two cases the range of exponents $p$ and $p_n$ is not the full Banach range, due to the constraints in Theorem \ref{thm:main_twist}.

Finally, the mixed frequency component $T^{\mathrm{MH}}$ is new. Its analysis is similar to that of $T^{\mathrm{H}}$ and it is carried out in \S \ref{subsec:mixedhigh}. In particular, it relies on a fiberwise application of the smoothing effect of multilinear averages along curves described in \cite{CDR21}.

\subsection{Low frequencies \texorpdfstring{$\mathrm{L}$}{}} \label{subsec:low}

We have
\begin{equation*}
T^{\mathrm{L}} (\mathbf{f}) = \sum_{j \in {{\mathbb Z}}} T_j ( f_1 \ast_1 \widecheck{\varphi_{j}}, f_2 \ast_2 \widecheck{\varphi_{2j}}, f_3 \ast_3 \widecheck{\varphi_{3j}}).
\end{equation*}
Expanding the convolutions, we obtain the operator
\begin{equation*}
\int_{{{\mathbb R}}^3} f_1 (u,y,z) f_2 (x,v,z) f_3 (x,y,w) K_{\mathrm{L}} ( x-u, y-v, z-w ) du dv dw,
\end{equation*}
where the kernel $K_{\mathrm{L}}$ is defined by
\begin{equation*}
K_{\mathrm{L}} (u,v,w) = \sum_{j \in {{\mathbb Z}}} 2^{6j} \kappa_{\mathrm{L}} (2^{j} u, 2^{2j} v, 2^{3j} w) ,
\end{equation*}
for the function
\begin{equation*}
\kappa_{\mathrm{L}} (u,v,w) = \int_{{{\mathbb R}}} \widecheck{\varphi} (u+t) \widecheck{\varphi} (v+t^2) \widecheck{\varphi} (w+t^3) \psi(t) \frac{dt}{t}.
\end{equation*}
In particular, the function $\kappa_{\mathrm{L}}$ is a Schwartz function satisfying
\begin{equation*}
\widehat{\kappa_{\mathrm{L}}} (0,0,0) = \int_{{{\mathbb R}}^3} \kappa_{\mathrm{L}} (u,v,w) du dv dw = 0
\end{equation*}
because the function $\psi(t) t^{-1}$ has integral zero. Therefore, the Fourier multiplier $\widehat{K_{\mathrm{L}}}$ associated with the kernel $K_{\mathrm{L}}$ satisfies the anisotropic symbol condition in \eqref{eq:symbol_condition}.

By Theorem \ref{thm:main_twist}, for all $p \in (1, \tfrac{4}{3} )$, $p_1,p_2,p_3\in (1,4)$,
\begin{equation*}
\|T^{\mathrm{L}}(\mathbf{f})\|_{p} \le C\prod_{n=1}^3 \|f_n\|_{p_n}.
\end{equation*}

\subsection{Mixed frequencies \texorpdfstring{$\mathrm{ML}$}{}} \label{subsec:mixedlow}

Without loss of generality, we restrict to the case of the operator $T^{\mathrm{ML},1}$. The remaining two cases of the operators $T^{\mathrm{ML},2}$ and $T^{\mathrm{ML},3}$ are dealt with in the same way.

We have
\begin{equation*}
T^{\mathrm{ML},1} (\mathbf{f}) = \sum_{j \in {{\mathbb Z}}} \sum_{k > 0} T_j ( f_1 \ast_1 \widecheck{\psi_{j+k}}, f_2 \ast_2 \widecheck{\varphi_{2j+k-101}}, f_3 \ast_3 \widecheck{\varphi_{3j+k-101}}).
\end{equation*}
Expanding the convolutions, we obtain the operator
\begin{equation*}
\int_{{{\mathbb R}}^3} f_1 (u,y,z) f_2 (x,v,z) f_3 (x,y,w) K_{\mathrm{ML},1} ( x-u, y-v, z-w ) du dv dw,
\end{equation*}
where the kernel $K_{\mathrm{ML},1}$ is defined by
\begin{equation*}
K_{\mathrm{ML},1} (u,v,w) = \sum_{j \in {{\mathbb Z}}} 2^{6j} \kappa_{\mathrm{ML},1} (2^{j} u, 2^{2j} v, 2^{3j} w) ,
\end{equation*}
for the function
\begin{equation*}
\kappa_{\mathrm{ML},1} (u,v,w) = \sum_{k > 0} \int_{{{\mathbb R}}} \widecheck{\psi_k} (u+t) \widecheck{\varphi_{k-101}} (v+t^2) \widecheck{\varphi_{k-101}} (w+t^3) \psi(t) \frac{dt}{t}.
\end{equation*}

The Fourier transform $\widehat{\kappa_{\mathrm{ML},1}}$ of the function $\kappa_{\mathrm{ML},1}$ is
\begin{equation*}
\widehat{\kappa_{\mathrm{ML},1}} (\xi, \eta, \zeta) = \sum_{k > 0} \psi (2^{-k} \xi) \varphi (2^{-k+101} \eta ) \varphi (2^{-k+101} \zeta ) \int_{{{\mathbb R}}} e^{2 \pi i (t \xi + t^2 \eta + t^3 \zeta)} \psi(t) \frac{dt}{t},
\end{equation*}
and it vanishes on a neighbourhood of the origin, hence
\begin{equation*}
\widehat{\kappa_{\mathrm{ML},1}} (0,0,0) = 0.
\end{equation*}

Furthermore, for every $\xi = (\xi_1, \xi_2, \xi_3) \in {{\mathbb R}}^3$ in the support of $\widehat{\kappa}_{\mathrm{ML},1}$ and every $t \in {{\mathbb R}}$ in the support of $\psi$, the derivative $\xi_1 + 2t \xi_2 + 3t^2 \xi_3$ of the phase does not vanish, hence we can apply integration by parts repeatedly to the function
\begin{equation*}
\theta (\xi) = \int_{{{\mathbb R}}} e^{2 \pi i (t \xi_1 + t^2 \xi_2 + t^3 \xi_3)} \psi(t) \frac{dt}{t},
\end{equation*}
to obtain, for every triple $\alpha = (\alpha_1,\alpha_2,\alpha_3)$ of positive integers, every $\xi$ in the support of $\kappa_{\mathrm{ML},1}$, and every $N \geq 0$, the estimate
\begin{equation*}
| \partial_{\xi}^\alpha \theta (\xi) | \leq C(\alpha, N) ( 1 + |\xi| )^{-N}.
\end{equation*}

In particular, the function
$\kappa_{\mathrm{ML},1}$ is Schwartz. Therefore, the Fourier multiplier $\widehat{K}_{\mathrm{ML},1}$ associated with the kernel $K_{\mathrm{ML},1}$ satisfies the anisotropic symbol condition in \eqref{eq:symbol_condition}.

Again, by Theorem \ref{thm:main_twist}, for all $p \in (1, \tfrac{4}{3} )$, $p_1,p_2,p_3\in (1,4)$,
\begin{equation*}
\|T^{\mathrm{ML},1}(\mathbf{f})\|_{p} \le C\prod_{n=1}^3 \|f_n\|_{p_n}.
\end{equation*}

\subsection{High frequencies} \label{subsec:high}

We claim that for all $p \in [1,\infty)$, $p_1, p_2, p_3 \in (1,\infty)$ satisfying the homogeneity condition
\begin{equation} \label{eq:homogeneity}
\frac{1}{p_1} + \frac{1}{p_2} + \frac{1}{p_3} = \frac{1}{p} ,
\end{equation}
we have
\begin{equation} \label{eq:final_gain_H}
\| T^{\mathrm{H}} (\mathbf{f}) \|_{p} \lesssim \prod_{n=1}^3 \|f_n\|_{p_n}.
\end{equation}
To obtain this result, we interpolate between two different estimates for the operator $T^{(k)}$ defined in \eqref{eq:decomposition_k}.
The first is an estimate with exponential gain in $k \in \mathfrak{F}_{\mathrm{H}}$, the second is an estimate with polynomial loss in $k \in \mathfrak{F}_{\mathrm{H}}$.
\begin{lemma} \label{lem:gain}
Let $p \in [1,\infty)$, $p_1, p_2, p_3 \in [2,\infty)$ satisfying the homogeneity condition in \eqref{eq:homogeneity}.

Then there exist constants $C = C(p_1,p_2,p_3,p) < \infty$ and $\delta = \delta(p_1,p_2,p_3,p) > 0$ such that, for every $k \in \mathfrak{F}_{\mathrm{H}}$ and every triple $\mathbf{f}$ of Schwartz functions,
\begin{equation} \label{eq:gain}
\| T^{(k)}(\mathbf{f}) \|_{p} \leq C 2^{- \delta |k|} \prod_{n=1}^3 \|f_n\|_{p_n} .
\end{equation}
\end{lemma}
\begin{lemma} \label{lem:loss}
Let $p \in [1,\infty)$, $p_1, p_2, p_3 \in (1,\infty)$ satisfying the homogeneity condition in \eqref{eq:homogeneity}.

Then there exists a constant $C = C(p_1,p_2,p_3,p) < \infty$ such that, for every $k \in \mathfrak{F}_{\mathrm{H}}$ and every triple $\mathbf{f}$ of Schwartz functions,
\begin{equation} \label{eq:loss}
\| T^{(k)}(\mathbf{f}) \|_{p} \leq C |k|^5 \prod_{n=1}^3 \|f_n\|_{p_n} .
\end{equation}
\end{lemma}
Interpolating between the two types of bound, we obtain an estimate with exponential gain in $k \in \mathfrak{F}_{\mathrm{H}}$. Summing over $k \in \mathfrak{F}_{\mathrm{H}}$, we recover the desired estimate in \eqref{eq:final_gain_H}.

\begin{proof}[Proof of Lemma \ref{lem:gain}]

By the definition of $\widetilde{\Delta}$ and Theorem \ref{thm:smoothing}, there exist constants $C < \infty$ and $\delta > 0$ such that, for all $j \in {{\mathbb Z}}$ and $k \in \mathfrak{F}_{\mathrm{H}}$,
\begin{equation*}
\| T_j (\Delta^{(1)}_{j+k_1} f_1, \Delta^{(2)}_{2j+k_2} f_2, \Delta^{(3)}_{3j+k_3} f_3) \|_{p} \leq C 2^{- \delta |k|} \prod_{n=1}^3 \| \widetilde{\Delta}^{(n)}_{nj+k_n} f_n \|_{p_n} .
\end{equation*}
Therefore, by H\"{o}lder's inequality with exponents $(p_1,p_2,p_3)$ applied to the sum in $j \in {{\mathbb Z}}$, we have that, for every $k \in \mathfrak{F}_{\mathrm{H}}$,
\begin{equation*}
\| T^{(k)} (\mathbf{f}) \|_{1} \le C 2^{- \delta |k|} \prod_{n=1}^3 \| ( \sum_{j \in {{\mathbb Z}}} | \widetilde{\Delta}^{(n)}_{nj+k_n} f_n |^{p_n} )^{\frac{1}{p_n}} \|_{p_n}.
\end{equation*}
By the embedding between $\ell^p$ spaces and the $L^p$ estimate for the Littlewood--Paley square function, for every $n \in \{1,2,3\}$,
\begin{equation*}
\| ( \sum_{j \in {{\mathbb Z}}} | \widetilde{\Delta}^{(n)}_{nj+k_n} f_n |^{p_n} )^{\frac{1}{p_n}} \|_{p_n} \leq \| ( \sum_{j \in {{\mathbb Z}}} | \widetilde{\Delta}^{(n)}_{nj+k_n} f_n |^{2} )^{\frac{1}{2}} \|_{p_n} \lesssim \| f_n \|_{3}.
\end{equation*}
The inequalities in the previous two displays yield the desired estimate in \eqref{eq:gain}.
\end{proof}

\begin{proof}[Proof of Lemma \ref{lem:loss}]

We introduce an auxiliary operator, the shifted dyadic maximal function $\mathcal{M}_\sigma$. For every $\sigma \in {{\mathbb R}}$, we define $\mathcal{M}_\sigma$ by, for every Schwartz function $g \in \mathcal{S}({{\mathbb R}})$,
\begin{equation}
\mathcal{M}_{\sigma} g (x) = \sup_{j \in {{\mathbb Z}}} 2^{-j} \int_{[\sigma 2^{j}, (\sigma+1) 2^{j} )} | g(x+t) | dt .
\end{equation}

We claim the following result.

\begin{lemma} \label{lem:shifted_max}
There exists $C < \infty$ such that for all $ \kappa, \lambda \in {{\mathbb Z}}$, $ \kappa, \lambda > 0$, $| \kappa - \lambda | < 100$, we have the pointwise estimates
\begin{equation} \label{eq:est_shifted_max}
| T_0( f_1, \Delta^{(2)}_\kappa f_2, \Delta^{(3)}_\lambda f_3) | \lesssim \sum_{\iota \in \mathcal{I}^{(3)}} a_\iota (\mathcal{M}^{(1)}_{\sigma_{1,\iota}} f_1) (\mathcal{M}^{(2)}_{\sigma_{2,\iota}} f_2) (\mathcal{M}^{(3)}_{\sigma_{3,\iota}} f_3),
\end{equation}
where $\mathcal{I}$ is a countable set and $a_\iota > 0$, $\sigma_{1,\iota}, \sigma_{2,\iota} , \sigma_{3,\iota} \in {{\mathbb R}}$ satisfy, for all $b_1,b_2,b_3 > 0$,
\begin{equation} \label{eq:summability}
\sum_{\iota \in \mathcal{I}} a_\iota \log(2 + | \sigma_{1,\iota} | )^{b_1} \log(2 + | \sigma_{2,\iota} | )^{b_2} \log(2 + | \sigma_{3,\iota} | )^{b_3} \leq C \kappa^{b_1+b_2+b_3} \simeq C \lambda^{b_1+b_2+b_3} .
\end{equation}
Analogous results hold true for the operators
\begin{equation*}
T_0( \Delta^{(1)}_\kappa f_1, \Delta^{(2)}_\lambda f_2, f_3), \qquad T_0( \Delta^{(1)}_\kappa f_1, f_2, \Delta^{(3)}_\lambda f_3).
\end{equation*}
\end{lemma}
\begin{proof}
Without loss of generality, we assume $\kappa \geq \lambda$. The proof in the case of $\lambda \geq \kappa$ is analogous.
We have
\begin{align*}
& T_0(f_1, \Delta^{(2)}_\kappa f_2, \Delta^{(3)}_\lambda f_3) \\
& = \int_{{{\mathbb R}}^3} f_1 (x+t,y,z) f_2 (x,y+t^2-u,z) f_3 (x,y,z+t^3-v) \widecheck{\psi_\kappa} (u) \widecheck{\psi_\lambda} (v)
\psi(t) \frac{dt}{t} du dv .
\end{align*}

For every $l \in {{\mathbb Z}}$, we define the dyadic interval
\begin{equation*}
I_l = [l 2^{-\kappa}, (l+1) 2^{-\kappa}) = l 2^{-\kappa} + [0,2^{-\kappa}).
\end{equation*}
We split the integration in $t$, $u$, and $v$ over these intervals, so that by the rapid decay of $\widecheck{\psi}$, we bound the left hand side in \eqref{eq:est_shifted_max} by
\begin{align*}
& \sum_{m, n \in {{\mathbb Z}}} (1 + |m|)^{-N} (1 + |n|)^{-N} 2^{-\kappa} \\
& \sum_{2^{\kappa-1} \leq |l| \leq 2^{\kappa + 1} } 2^{3\kappa} \int_{I_m} \int_{I_n} \int_{I_l} | f_1 (x+t,y,z) | | f_2 (x,y+t^2-u,z) | | f_3 (x,y,z+t^3-v) | dt du dv .
\end{align*}

If $l > 0$, for all $m,n \in {{\mathbb Z}}$, we bound the argument of the sum in $l$ by
\begin{equation*}
2^{3\kappa} \int_{I_l} \int_{J_{l,m}} \int_{K_{l,n}} | f_1 (x+t,y,z) | | f_2 (x,y-u,z) | | f_3 (x,y,z-v) | du dv dt,
\end{equation*}
where
\begin{align*}
& I_m - t^2
\subseteq J_{l,m} = (m 2^{-3} - (l+1)^2 2^{-\kappa-3} )2^{-\kappa +3} + [ 0, 2^{-\kappa+3} ) , \\
& I_n - t^3
\subseteq K_{l,n} = (n 2^{-5} - (l+1)^3 2^{-2\kappa-5} )2^{-\kappa + 5} + [ 0, 2^{-\kappa+5} ) .
\end{align*}
Therefore, for
\begin{equation*}
\sigma_{1} = \sigma_{1,l} = l, \quad \sigma_{2} = \sigma_{2,l,m} = m 2^{-3} - (l+1)^2 2^{-\kappa-3}, \quad \sigma_{3} = \sigma_{3,l,n} = n 2^{-5} - (l+1)^3 2^{-2\kappa-5},
\end{equation*}
we bound the argument of the sum in $l$ by
\begin{equation*}
\mathcal{M}^{(1)}_{\sigma_{1}} f_1 (x,y,z) \mathcal{M}^{(2)}_{\sigma_{2}} f_2 (x,y,z) \mathcal{M}^{(3)}_{\sigma_{3}} f_3 (x,y,z).
\end{equation*}
If $l < 0$, for all $m,n \in {{\mathbb Z}}$ we obtain an analogous bound, yielding the desired estimate in \eqref{eq:est_shifted_max}. Note that the summability condition on the coefficients in \eqref{eq:summability} is satisfied.
\end{proof}

By dilation invariance, Lemma \ref{lem:shifted_max} implies that for every $j \in {{\mathbb Z}}$ and every $k \in \mathfrak{F}_{\mathrm{H}}$ (e.g. $k^{(3)} = k_1$),
\begin{equation} \label{eq:single_scale}
| T_j (\Delta^{(1)}_{j + k_1} f_1, \Delta^{(2)}_{2j + k_2} f_2, \Delta^{(3)}_{3j + k_3} f_3) | \lesssim \sum_{\iota \in \mathcal{I}} a_\iota (\mathcal{M}^{(1)}_{\sigma_{1,\iota}} \Delta^{(1)}_{j + k_1} f_1) (\mathcal{M}^{(2)}_{\sigma_{2,\iota}} f_2) (\mathcal{M}^{(3)}_{\sigma_{3,\iota}} f_3),
\end{equation}
where we used the identity
\begin{align*}
& | T_j (\Delta^{(1)}_{j + k_1} f_1, \Delta^{(2)}_{2j + k_2} f_2, \Delta^{(3)}_{3j + k_3} f_3) | \\
& = | D_{(2^j,2^{2j},2^{3j})} T_0 ( D_{(2^{-j},2^{-2j},2^{-3j})} \Delta^{(1)}_{j + k_1} f_1, \Delta^{(2)}_{k_2} D_{(2^{-j},2^{-2j},2^{-3j})} f_2, \Delta^{(3)}_{k_3} D_{(2^{-j},2^{-2j},2^{-3j})} f_3) | ,
\end{align*}
with $D_{(a,b,c)} f(x,y,z) = f(ax,by,cz)$, $D_{a} g(x) = g(ax)$, and $D_{a} \mathcal{M} D_{1/a} = \mathcal{M}$.

By the definition of $\widetilde{\Delta}$ and the estimate in \eqref{eq:single_scale},, for every $k \in \mathfrak{F}_{\mathrm{H}}$,
\begin{equation*}
\| T^{(k)} (\mathbf{f}) \|_{p} \lesssim \sum_{\iota \in \mathcal{I}} a_\iota \| \sum_{j \in {{\mathbb Z}}} (\mathcal{M}^{(1)}_{\sigma_{1,\iota}} \Delta^{(1)}_{j + k_1} f_1) (\mathcal{M}^{(2)}_{\sigma_{2,\iota}} \widetilde{\Delta}^{(2)}_{2j + k_2} f_2) (\mathcal{M}^{(3)}_{\sigma_{3,\iota}} \widetilde{\Delta}^{(3)}_{3j + k_3} f_3) \|_{p}.
\end{equation*}
By H\"{o}lder's inequality with exponents $(2,2,\infty)$ applied to the sum in $j \in {{\mathbb Z}}$ and with exponents $(p_1,p_2,p_3)$ applied to the $L^p$ norm, we bound the right hand side in the previous display by
\begin{equation*}
\sum_{\iota \in \mathcal{I}} a_\iota \| ( \sup_{j \in {{\mathbb Z}}} \mathcal{M}^{(1)}_{\sigma_{1,\iota}} \Delta^{(1)}_{j + k_1} f_1 ) \|_{p_1} \| ( \sum_{j \in {{\mathbb Z}}} | \mathcal{M}^{(2)}_{\sigma_{2,\iota}} \widetilde{\Delta}^{(2)}_{2j + k_2} f_2 |^2 )^{\frac{1}{2}} \|_{p_2} \| ( \sum_{j \in {{\mathbb Z}}} | \mathcal{M}^{(3)}_{\sigma_{3,\iota}} \widetilde{\Delta}^{(3)}_{3j + k_3} f_3 |^2 )^{\frac{1}{2}} ) \|_{p_3}.
\end{equation*}
Next, we observe that there exists a constant $C< \infty$ such that, for every $j \in {{\mathbb Z}}$, we have $| \Delta_j g | \leq C \mathcal{M} g $. Moreover, we recall that, for every $p \in (1,\infty)$,
\begin{equation*}
\| \mathcal{M}_\sigma g \|_{p} \lesssim \log(2 + | \sigma | )^{\frac{1}{p}} \| g \|_{p} , \qquad \| ( \sum_{j \in {{\mathbb Z}}} | \mathcal{M}_\sigma g_j |^2 )^{\frac{1}{2}} \|_{p} \lesssim \log(2 + | \sigma | )^2 \| ( \sum_{j \in {{\mathbb Z}}} | g_j |^2 )^{\frac{1}{2}} \|_{p} .
\end{equation*}
For a proof of the second inequality in the last display, see Theorem 3.1 in \cite{GHLR17}.

Together with the estimate for the Hardy--Littlewood maximal function and the Littlewood--Paley square function, we continue the chain of inequalities with
\begin{equation*}
\sum_{\iota \in \mathcal{I}} a_\iota \log(2 + | \sigma_{1,\iota} | )^{\frac{1}{p_1}} \log(2 + | \sigma_{2,\iota} | )^2 \log(2 + | \sigma_{3,\iota} | )^2 \| f_1 \|_{p_1} \| f_2 \|_{p_2} \| f_3 \|_{p_3}.
\end{equation*}
The summability condition on the coefficients in \eqref{eq:summability} yield the desired estimate in \eqref{eq:loss}.
\end{proof}

\subsection{Mixed frequencies \texorpdfstring{$\mathrm{MH}$}{}} \label{subsec:mixedhigh}

Without loss of generality, we restrict to the case of the operator $T^{\mathrm{MH},1}$. The remaining two cases of the operators $T^{\mathrm{MH},2}$ and $T^{\mathrm{MH},3}$ are dealt with in the same way.

We claim that for all $p \in [1,\infty)$, $p_1, p_2, p_3 \in (1,\infty)$ satisfying the homogeneity condition in \eqref{eq:homogeneity} we have
\begin{equation} \label{eq:final_gain_MH}
\| T^{\mathrm{MH},1} (\mathbf{f}) \|_{p} \lesssim \| f_1 \|_{p_1} \| f_2 \|_{p_2} \| f_3 \|_{p_3}.
\end{equation}
To obtain this result, we interpolate between two different estimates for the operator $T^{(k)}$ defined in \eqref{eq:decomposition_k} and a slight variation of it, namely the auxiliary operator T$^{(k_2,k_3)}$ defined by
\begin{equation*}
T^{(k_2,k_3)}(\mathbf{f}) = \sum_{j \in {{\mathbb Z}}} T_j(S^{(1)}_{j} f_1, \Delta^{(2)}_{2j+k_2} f_2, \Delta^{(3)}_{3j+k_3} f_3).
\end{equation*}
The first is an estimate with exponential gain in $k \in \mathfrak{F}_{\mathrm{MH},1}$, the second is an estimate with polynomial loss in $k \in \mathfrak{F}_{\mathrm{MH},1}$.

In particular, to obtain the first estimate, we split further $\mathfrak{F}_{\mathrm{MH},1} = \mathfrak{F}_{\mathrm{MH},1,+} \cup \mathfrak{F}_{\mathrm{MH},1,-}$, where we define
\begin{align*}
& \mathfrak{F}_{\mathrm{MH},1,+} = \{ k \in \mathfrak{F}_{\mathrm{MH},1} \colon 0 < k_1 < k^{(2)} \}, \\
& \mathfrak{F}_{\mathrm{MH},1,+} = \{ k \in \mathfrak{F}_{\mathrm{MH},1} \colon k_1 \leq 0 \}.
\end{align*}
Note that, for all fixed $k_2,k_3 \in {{\mathbb Z}}$ such that $\max(k_2,k_3) > 0$ and $|k_2 - k_3| \leq 100$,
\begin{equation*}
T^{(k_2,k_3)}(\mathbf{f}) = \sum_{k_1 \colon (k_1,k_2,k_3) \in \mathfrak{F}_{\mathrm{MH},1,-}} T^{(k)}(\mathbf{f}).
\end{equation*}

\begin{lemma} \label{lem:gain_2+}
Let $p \in [1,\infty)$, $p_1, p_2, p_3 \in [2,\infty)$ satisfying the homogeneity condition \eqref{eq:homogeneity}.
Then there exist constants $C = C(p_1,p_2,p_3,p) < \infty$ and $\delta = \delta(p_1,p_2,p_3,p) > 0$ such that, for every $k \in \mathfrak{F}_{\mathrm{MH},1,+}$ and every triple $\mathbf{f}$ of Schwartz functions,
\begin{equation*}
\| T^{(k)}(\mathbf{f}) \|_{p} \leq C 2^{- \delta |k|} \prod_{n=1}^3 \|f_n\|_{p_n} .
\end{equation*}
\end{lemma}
\begin{lemma} \label{lem:gain_2-}
There exist constants $C < \infty$ and $\delta > 0$ such that, for all $k_2,k_3 \in {{\mathbb Z}}$ such that $\max(k_2,k_3) > 0$ and $|k_2 - k_3| \leq 100$, and every triple $\mathbf{f}$ of Schwartz functions,
\begin{equation} \label{eq:gain_2-}
\| T^{(k_2,k_3)}(\mathbf{f}) \|_{1} \leq C 2^{- \delta |\max(k_2,k_3)|} \|f_1\|_{\infty} \|f_2\|_{2} \|f_3\|_{2}.
\end{equation}
\end{lemma}
\begin{rem}
It is worth noting that the triple $(\infty,2,2)$ or any of its permutations do not belong to the range of exponents in Theorem \ref{thm:KMPWW}, hence in Theorem \ref{thm:smoothing}.
\end{rem}
\begin{lemma} \label{lem:loss_2}
Let $p \in [1,\infty)$, $p_1, p_2, p_3 \in (1,\infty)$ satisfying the homogeneity condition in \eqref{eq:homogeneity}.

Then there exists a constant $C = C(p_1,p_2,p_3,p) < \infty$ such that, for all $k_2,k_3 \in {{\mathbb Z}}$ such that $\max(k_2,k_3) > 0$ and $|k_2 - k_3| \leq 100$, and every triple $\mathbf{f}$ of Schwartz functions,
\begin{equation*}
\| T^{(\max(k_2,k_3) - 100,k_2,k_3)}(\mathbf{f}) \|_{p} \leq C |k|^5 \prod_{n=1}^3 \|f_n\|_{p_n} .
\end{equation*}
\end{lemma}
Interpolating between the two types of bound, we obtain an estimate with exponential gain in $k \in \mathfrak{F}_{\mathrm{MH},1}$. Summing over $k \in \mathfrak{F}_{\mathrm{MH},1}$, we recover the desired estimate in \eqref{eq:final_gain_MH}. Note that for all fixed $k_2,k_3 \in {{\mathbb Z}}$, $k_2,k_3 > 100$ such that $|k_2 - k_3| \leq 100$ we have $(\max\{ k_2, k_3 \} - 100)$-many values of $k_1$ such that $k \in \mathfrak{F}_{\mathrm{MH},1,+}$, hence linearly many. This is different from the case of $\mathfrak{F}_{\mathrm{H}}$, where for all fixed $k_2,k_3 \in {{\mathbb Z}}$, $k_2,k_3 > 100$ such that $|k_2 - k_3| \leq 100$ we only have $100$-many values of $k_1$ such that $k \in \mathfrak{F}_{\mathrm{H}}$. However, the exponential gain $k \in \mathfrak{F}_{\mathrm{MH},1,+}$ absorbs the linear term, recovering summability in $k$.

\begin{proof}[Proof of Lemma \ref{lem:gain_2+}]
The proof is analogous to that of Lemma \ref{lem:gain}. Note that, for every $\delta > 0$ and every $k \in \mathfrak{F}_{\mathrm{MH},1,+}$, $2^{-\delta |k|} \simeq 2^{-\delta k^{(1)}}$.
\end{proof}

\begin{proof}[Proof of Lemma \ref{lem:gain_2-}]
We claim that it is enough to prove that, for every $j \in {{\mathbb Z}}$,
\begin{equation} \label{eq:mixedfreqpenult}
\| T_j(S^{(1)}_{j} f_1, \Delta^{(2)}_{2j+k_2} f_2, \Delta^{(3)}_{3j+k_3} f_3) \|_{1} \le C 2^{- \delta |\max(k_2,k_3)|} \|f_1\|_{\infty} \|f_2\|_{2} \|f_3\|_{2}.
\end{equation}
In fact, by the definition of $\widetilde{\Delta}$, the estimate in the previous display, and H\"{o}lder's inequality with exponents $(\infty,2,2)$ applied to the sum in $j \in {{\mathbb Z}}$, we have that, for all $k_2,k_3 \in {{\mathbb Z}}$ such that $\max(k_2,k_3) > 0$ and $|k_2 - k_3| \leq 100$,
\begin{equation*}
\| T^{(k_2,k_3)} (\mathbf{f}) \|_{1} \le C 2^{- \delta |\max(k_2,k_3)|} \sup_{j \in {{\mathbb Z}}} \| S_j^{(1)} f_1\|_{\infty} ( \sum_{j \in {{\mathbb Z}}} \| \widetilde{\Delta}^{(2)}_{2j+k_2} f_2\|^2_{2} )^{\frac{1}{2}} ( \sum_{j \in {{\mathbb Z}}} \| \widetilde{\Delta}^{(3)}_{3j+k_3} f_3 \|^2_{2} )^{\frac{1}{2}}.
\end{equation*}
Observing that there exists a constant $C< \infty$ such that, for every $j \in {{\mathbb Z}}$, $| S_j g | \leq C \mathcal{M} g $, the estimate for the Hardy--Littlewood maximal function and Plancherel identity yield the desired estimate in \eqref{eq:gain_2-}.

It only remains to show \eqref{eq:mixedfreqpenult}. By dilation invariance, we assume $j=0$. We have
\begin{align*}
& | T_0( S^{(1)}_0 f_1, \Delta^{(2)}_{k_2} f_2, \Delta^{(3)}_{k_3} f_3) | \\
& \le \int_{{{\mathbb R}}} | f_1 (x+s,y,z) | \Big| \int_{{{\mathbb R}}} \Delta^{(2)}_{k_2} f_2 (x,y+t^2,z) \Delta^{(3)}_{k_3} f_3 (x,y,z+t^3) \widecheck{\varphi} (t-s)
\psi(t) \frac{dt}{t} \Big| ds .
\end{align*}
By Fubini's theorem,
\begin{equation*}
\| T_0( S^{(1)}_0 f_1, \Delta^{(2)}_{k_2} f_2, \Delta^{(3)}_{k_3} f_3) \|_{1} \leq \| f_1 \|_{\infty} \int_{{{\mathbb R}}} \int_{{{\mathbb R}}} \Big\| \int_{{{\mathbb R}}} g_{x}(y+t^2,z) h_{x} (y,z+t^3) \rho_s(t) dt \Big\|_{L^1({{\mathbb R}}^2)} dx ds,
\end{equation*}
where $\rho_s(t) = \psi(t) t^{-1} \widecheck{\phi}(t-s)$, $g_{x} (y,z) = \Delta^{(2)}_{k_2} f_2 (x,y,z) $, and $h_{x} (y,z) = \Delta^{(3)}_{k_3} f_3 (x,y,z) $. By Theorem \ref{thm:smoothing_2},
\begin{equation*}
\Big\| \int_{{{\mathbb R}}} g_{x}(y+t^2,z) h_{x} (y,z+t^3) \rho_s(t) dt \Big\|_{L^1({{\mathbb R}}^2)} \lesssim ( \| \rho_s \|_{\infty} + \| \rho'_s \|_{1} ) 2^{- \delta |\max(k_2,k_3)|} \|g_{x}\|_{2} \|h_{x}\|_{2}.
\end{equation*}
By Cauchy--Schwarz inequality applied to the integral in $x$, the support of $\psi$, and the rapid decay of $\widecheck{\varphi}$, we obtain the desired estimate in \eqref{eq:mixedfreqpenult}.
\end{proof}

\begin{proof}[Proof of Lemma \ref{lem:loss_2}]
The proof is analogous to that of Lemma \ref{lem:loss}. We only need to replace $\Delta^{(1)}_{j+k_1}$ by $S^{(1)}_{j+k_1}$ and observe that there exists a constant $C< \infty$ such that, for every $j \in {{\mathbb Z}}$, we have$| S_j g | \leq C \mathcal{M} g $.
\end{proof}

\section{Smoothing inequalities} \label{sec:smoothing}

We start by restating the smoothing inequality from \cite[Thm. 6.1]{KMPWW24} in the special case we are concerned with.
For $N > 0$, let $A_{N}$ be the averaging operator defined by, for every triple $\mathbf{F} =(F_1,F_2,F_3)$ of Schwartz functions $F_1,F_2,F_3 \in \mathcal{S}({{\mathbb R}}^3)$,
\begin{equation}\label{eq:KMPW-real-average}
A_{N}(\mathbf{F})(x)
=N^{-1}\int_0^N \prod_{n=1}^3 F_n(x+t^n e_n) dt.
\end{equation}
\begin{theorem}[Kosz-Mirek-Peluse-Wan-Wright, special case of Thm. 6.1 in \cite{KMPWW24}]\label{thm:KMPWW}
Let $p \in [1,\infty)$, $p_1, p_2, p_3 \in (1,\infty)$ satisfying the homogeneity condition in \eqref{eq:homogeneity}, and let $C > 0$.
Then there exists $\delta = \delta (p,p_1,p_2,p_3,C) \in(0,1)$ such that the following holds. Let $N \geq 1$ and suppose the $j$-th partial Fourier transform of $F_j$ vanishes on
$\{\xi \colon |\xi_n|\leq N^{C-n} \}$ for at least one $n \in \{1,2,3\}$. Then, for every triple $\mathbf{F}$ of Schwartz functions,
\begin{equation}\label{eq:KMPW-bound}
\|A_{N}(F)\|_{p}
\lesssim N^{-c} \prod_{n=1}^3\|F_n\|_{p_n}.
\end{equation}
\end{theorem}

\begin{proof}[Proof of Theorem \ref{thm:smoothing}]
For $\lambda < 1/a$, the desired estimate in \eqref{eq:smoothing} follows by Minkowski's and H\"{o}lder's inequalities.

Let $\lambda \geq 1/a$. For every $x \in {{\mathbb R}}^3$, let $D_\lambda x = (\lambda x_1, \lambda^2 x_2, \lambda^3 x_3)$, and for every triple $\mathbf{f}$ of Schwartz functions, let $(D_\lambda \mathbf{f}) (x) = \mathbf{f} (D_{\lambda^{-1}} x)$. We have
\begin{equation*}
\mathcal{A}_{\psi}(\mathbf{f})(x) = \mathcal{A}_{\psi_\lambda} (D_\lambda \mathbf{f})(D_\lambda x) ,
\end{equation*}
where $\psi_\lambda ( t) = \lambda^{-1} \psi (\lambda^{-1} t)$, and by assumption $\supp (\widehat{D_{\lambda} f_n} ) \subseteq \{ | \xi_n | \geq \lambda^{1-n} \} $ for at least one $n \in \{1,2,3\}$.

Integrating by parts,
\begin{equation*}
\mathcal{A}_{\psi_\lambda} (D_\lambda \mathbf{f})(y) = -\int_0^\infty t \psi'_\lambda(t) \Big( \frac{1}{t} \int_0^t \prod_{n=1}^3 D_{\lambda} f_n(y+s^n e_n) ds\Big) dt,
\end{equation*}
where we used the information on the support of $\psi$, namely $\supp (\psi) \subseteq (0,\infty)$. By Minkowski's inequality,
\begin{equation*}
\| \mathcal{A}_{\psi_\lambda}(D_\lambda \mathbf{f}) \|_{p} \lesssim \| t \psi_\lambda \|_{1} \sup_{t \in [a\lambda, b \lambda]} \| A_{t} (D_\lambda \mathbf{f}) \|_{p} ,
\end{equation*}
and $\| t \psi'_\lambda \|_{1} = \| t \psi' \|_{1} < \infty $.

By Theorem \ref{thm:KMPWW} with $C=1$ and $N = t$, we obtain
\begin{equation*}
\sup_{t \in [a\lambda, b \lambda]} \| A_{t} (D_\lambda \mathbf{f}) \|_{p} \lesssim \lambda^{-\delta} \prod_{n=1}^3 \|D_\lambda f_n\|_{p_n}.
\end{equation*}
By the homogeneity condition in \eqref{eq:homogeneity} for the exponents $p, p_n$, we recover the desired estimate in \eqref{eq:smoothing}.
\end{proof}

The proof of Theorem \ref{thm:smoothing_2} is analogous to that of Theorem 5 in \cite{CDR21}, keeping track of the dependence on the cutoff function $\psi$. This is also the route taken in the formalized Lean proof available on \href{https://github.com/roos-j/lean-dfr}{GitHub}.
Alternatively, it can also be derived in the same manner as Theorem \ref{thm:smoothing} from \cite[Thm. 6.1]{KMPWW24}.

\section{Anisotropic twisted paraproduct}\label{sec:twisted}
To prove Theorem \ref{thm:main_twist}, by duality it suffices to show that given $p_0\in (4,\infty)$, $p_1,p_2,p_3\in (1,4)$ with $\sum_{n=0}^3p_n^{-1}=1$,
there exists a constant $C>0$ such that for all Schwartz functions $f_0,f_1,f_2,f_3$,
\[\Big |\int_{{{\mathbb R}}^3}\int_{{{\mathbb R}}^3}f_0(x)\prod_{n=1}^3f_n(x+t_ne_n)K(t)dx dt \Big |\le C\prod_{n=0}^3 \|f_n\|_{p_n}\]

The proof of this estimate closely follows the proof of \cite[Theorem 2]{CDR21}, see also the foundational works \cite{Kov12}, \cite{Ber12}.

The first step is a decomposition of $\widehat{K}$ into pieces localized to cones in frequency space. By symmetry, this reduces matters to the forms $\Lambda_u$ defined below, in which the third coordinate carries a cancellative kernel, while the first two coordinates carry non-cancellative kernels, see Proposition \ref{prop:fwcz}. A fiber-wise Calder{\'o}n--Zygmund decomposition then reduces Proposition \ref{prop:fwcz} to the range of exponents in Proposition \ref{prop:twisted}.

To prove Proposition \ref{prop:twisted}, we localize $\Lambda_u$ to convex trees of anisotropic dyadic boxes in the space of translation and scale parameters. On a fixed tree, one application of the Cauchy--Schwarz inequality separates $f_3$ from $f_0,f_1,f_2$. The form involving $f_3$ is a square function, which is bounded on $L^2$ by Plancherel's theorem, while the form involving $f_0,f_1,f_2$ is dominated by a superposition of local cube forms with Gaussian kernels.

The cube forms are bounded by repeated applications of the Cauchy--Schwarz inequality, which symmetrize the tuple of input functions, combined with a telescoping identity which allows one to trade the form with the cancellative kernel in a given coordinate for the forms with cancellative kernels in the remaining coordinates, up to a remainder term. The remainder is controlled by a Brascamp--Lieb inequality in terms of the local $L^4$ averages of the input functions over the tree, and is the main reason for the restrictions in the exponent range in Proposition \ref{prop:twisted}. Finally, a stopping time argument sums the resulting estimates over all trees and yields the bound in terms of the $L^{p_j}$ norms.

To perform a cone decomposition of $\widehat{K}$, let
\[\g(x) = e^{-\pi x^2},\quad \h(x) = \g'(x).\]
Let $\psi$ be a Schwartz function with $\widehat{\psi}$ supported in $[-2,-1]\cup [1,2]$ and $\varphi$ a Schwartz function with $\widehat{\varphi}$ supported in $[-2,2]$. Denote
\[\varphi_{t,p}(x) = t^{-1}\varphi(t^{-1}(x-p))\]
and $\varphi_t = \varphi_{t,0}$. Then
\[(\varphi_{1,r})_{t,q} = \varphi_{t,q+rt}\]

For $u=(u_1,u_2,u_3)\in {{\mathbb R}}^3$, $i\in \{1,2,3\}$, and a tuple $f=(f_0,f_1,f_2,f_3)$ of Schwartz functions, define
\[\Lambda_u^{(i)}(f) = \int_{0}^\infty c_i(t,u) \int_{{{\mathbb R}}^3}
f_0(x)(f_i*_i(\h*\h*{\psi}_{1,u_i})_{t^{\alpha_i}})(x) \prod_{n=1,n\ne i}^3 (f_n*_n(\g*{\varphi}_{1,u_n})_{t^{\alpha_n}})(x) dx \frac{dt}{t}
\]
where $c_i$ are measurable coefficients satisfying $|c_i(t,u)|\le 1$, and the subscripts $*_i$ and $*_n$ denote the convolutions in $i$-th and $n$-th coordinates, respectively.
We write \[\Lambda_u = \Lambda_u^{(3)}.\]

The choice of Gaussians in the definition of $\Lambda^{(i)}_u$ is due to the heat equation, which is the basis of the telescoping identity in Lemma \ref{lem:telescoping_main} below. Since Gaussians are moreover positive, the telescoping identity can be used as an inequality in the proof of Proposition~\ref{prop:fullcubetree} below. Similarly, Gaussians were used, for example, in \cite{Dur15}.

Throughout the paper we denote
\[\langle u \rangle = 1 + |u_1| + |u_2| + |u_3|.\]
By symmetry considerations and the cone decomposition arguments as in \cite{CDR21}, Theorem \ref{thm:main_twist} follows from the following proposition.
\begin{prop} \label{prop:fwcz}
Let $p_0,p_1,p_2,p_3\in (1,\infty)$ satisfy
\[\tfrac{1}{p_0}<\tfrac{1}{4},\quad \tfrac{1}{p_3}>\tfrac{1}{4}, \quad \tfrac{1}{p_1}+\tfrac{1}{p_2}>\tfrac{1}{4},\quad \tfrac{1}{p_2}+\tfrac{1}{p_3}>\tfrac{1}{2}, \quad \tfrac{1}{p_1}+\tfrac{1}{p_3}>\tfrac{1}{2},\quad \sum_{n=0}^3\tfrac{1}{p_n}=1.\] There exists a constant $C>0$ such that for all Schwartz tuples $f=(f_0,f_1,f_2,f_3)$, $u\in {{\mathbb R}}^3$, and all measurable coefficients $c_3(t,u)$ satisfying $|c_3(t,u)|\le1$,
\[|\Lambda_u(f)|\le C \langle u\rangle^{100} \prod_{n=0}^3 \|f_n\|_{p_n}\]
\end{prop}

By a fiber-wise Calder{\'o}n-Zygmund decomposition, Proposition \ref{prop:fwcz} will in turn be deduced from the following.
\begin{prop} \label{prop:twisted}
Let $p_0,p_1,p_2\in (4,\infty)$, $p_3\in [2,\infty)$, and $\sum_{n=0}^3\tfrac{1}{p_n}=1$. There exists a constant $C>0$ such that for all Schwartz tuples $f=(f_0,f_1,f_2,f_3)$, $u\in {{\mathbb R}}^3$, and all measurable coefficients $c_3(t,u)$ satisfying $|c_3(t,u)|\le1$,
\[|\Lambda_u(f)|\le C \langle u \rangle^{100}\prod_{n=0}^3 \|f_n\|_{p_n}\]
\end{prop}
The fiber-wise Calderón–Zygmund decomposition goes back to \cite{Ber12}. Its use here is a straightforward adaptation of the arguments in \cite[Section 4.5]{CDR21}. We omit the details.

Proposition \ref{prop:twisted} will be proved using twisted techniques, adapting the ideas from \cite{CDR21, DR21}. The rest of this section is devoted to the proof of Proposition \ref{prop:twisted}.

\subsection{Brascamp-Lieb estimates}
We will use the following estimates, which are the main reason for the restrictions in the range of exponents in Proposition \ref{prop:twisted}.
Let \[\theta(x)= (1+|x|)^{-10},\quad \;x\in{{\mathbb R}}\]
and for $t>0$ denote
\[\Theta_{t} = \theta_{t^{\alpha_1}}\otimes \theta_{t^{\alpha_2}}\otimes \theta_{t^{\alpha_3}} \]
Let $J=\{0,1\}^3$. For $j=(j_1,j_2,j_3)\in J$ and $x=(x_1^0,x_2^0,x_3^0,x_1^1,x_2^1,x_3^1)$, define
\[\Pi_j(x) = (x_1^{j_1},x_2^{j_2},x_3^{j_3}),\]
and write $p=(p_1,p_2,p_3)$.
\begin{lemma}[Full cube Brascamp-Lieb]
\label{lem:BL4}
For any $(p,t)\in {{\mathbb R}}^3 \times (0,\infty)$
and a tuple $(F_j)_{j\in J}$ of non-negative Schwartz functions,
\[\int_{{{\mathbb R}}^6}\Big( \prod_{j\in J} F_j(\Pi_j(x)) \Big) \prod_{i=1}^3\prod_{b=0}^1
\theta_{t^{\alpha_i},p_i}(x_i^b)
\, dx \le \prod_{j\in J} \Big( F_j^4*\Theta_t(p) \Big)^\frac{1}{4}\]
\end{lemma}
\begin{proof}
Denote
\begin{equation}
\label{measure_def}
d\mu(x)
=
\prod_{i=1}^3\prod_{b=0}^1d\mu_i^b(x_i^b), \quad d\mu_i^b(u)=\theta_{t^{\alpha_i},p_i}(u)\,du.
\end{equation}
For $j\in J$, let
\[
\overline{j}=(1,1,1)-j
\]
denote the vertex opposite to $j$. For
$j\in
J_0=\{0\}\times\{0,1\}^2$
let
\[
G_j(x)
=
F_j(\Pi_j(x))F_{\overline j}(\Pi_{\overline j}(x)).
\]
The left-hand side of Lemma \ref{lem:BL4} equals
\[
\int_{{{\mathbb R}}^6}
\prod_{j\in J_0} G_{j}(x)\,d\mu(x).
\]
By the Cauchy--Schwarz inequality, this is
\[
\leq
\Big(
\int_{{{\mathbb R}}^6}G_{(0,0,0)}(x)^2G_{(0,0,1)}(x)^2\,d\mu(x)
\Big)^{1/2}
\Big(
\int_{{{\mathbb R}}^6}G_{(0,1,0)}(x)^2G_{(0,1,1)}(x)^2\,d\mu(x)
\Big)^{1/2}.
\]
By the Cauchy--Schwarz inequality applied to the two factors on the
right-hand side, the last display is bounded by
\[
\prod_{j\in J_0}
\Big(
\int_{{{\mathbb R}}^6}G_j(x)^4\,d\mu(x)
\Big)^{1/4} = \prod_{j\in J_0}
(F_j^4*\Theta_t)(p)^{1/4}
(F_{\overline j}^4*\Theta_t)(p)^{1/4}
= \prod_{j\in J} (F_j^4*\Theta_t(p))^{1/4}.\]
\end{proof}
We also have the following improved estimate for the special case
in which all but two of the functions are identically equal to $1$.

\begin{lemma}[Edge Brascamp-Lieb]
\label{lem:BL2}
For any $(p,t)\in {{\mathbb R}}^3\times (0,\infty)$ and non-negative Schwartz functions $F_0,F_1$,
\[\int_{{{\mathbb R}}^4}F_0(x_1,x_2,x_3^0){F}_1(x_1,x_2,x_3^1) \theta_{t^{\alpha_1},p_1}(x_1) \theta_{t^{\alpha_2},p_2}(x_2) \]
\[\times \theta_{t^{\alpha_3},p_3}(x_3^0)\theta_{t^{\alpha_3},p_3}(x_3^1)
\, dx_1 dx_2 dx_3^0 dx_3^1 \le \big(F_0^2*\Theta_t(p)\big)^{1/2}
\big(F_1^2*\Theta_t(p)\big)^{1/2}. \]
\end{lemma}

\begin{proof}
For $b\in\{0,1\}$ denote
\[
G_b(x)
=
F_b(x_1,x_2,x_3^b)
\]
and let
\[d\mu(x)
=
d\mu_1^0(x_1)d\mu_2^0(x_2)
d\mu_3^0(x_3^0)d\mu_3^1(x_3^1)\]
where $d\mu^b_i$ are defined as in \eqref{measure_def}.

By the Cauchy--Schwarz inequality, we obtain for the form on the left-hand side
\[
\int_{{{\mathbb R}}^4}G_0(x)G_1(x)\,d\mu(x) \le \Big(
\int_{{{\mathbb R}}^4}G_0(x)^2\,d\mu(x)
\Big)^{1/2}
\Big(
\int_{{{\mathbb R}}^4}G_1(x)^2\,d\mu(x)
\Big)^{1/2}
\]
We compute the first factor. Since $G_0$ is independent of $x_3^1$,
Fubini's theorem gives
\[
\int_{{{\mathbb R}}^4}G_0(x)^2\,d\mu(x)
= \Big(
\int_{{{\mathbb R}}}
\theta_{t^{\alpha_3},p_3}(x_3^1)\,dx_3^1
\Big) \]
\[\times \int_{{{\mathbb R}}^3}
F_0(x_1,x_2,x_3^0)^2
\theta_{t^{\alpha_1},p_1}(x_1)
\theta_{t^{\alpha_2},p_2}(x_2)
\theta_{t^{\alpha_3},p_3}(x_3^0)
\,dx_1\,dx_2\,dx_3^0 \]
Since the integral of $\theta$ is at most one, this is bounded by
\[F_0^2*\Theta_t(p). \]
With a similar computation for the factor involving $G_1$ we obtain
\[\int_{{{\mathbb R}}^4}G_0(x)G_1(x)\,d\mu(x) \le (F_0^2*\Theta_t(p))^{1/2}
(F_1^2*\Theta_t(p))^{1/2}.\]
\end{proof}

\subsection{Local cube forms}
Proposition \ref{prop:twisted} will follow from an estimate for a local version of $\Lambda_u$, combined with a stopping time argument.
The first step towards local estimates for $\Lambda_u$ is to prove bounds for local cube forms.
We follow the localization setup from \cite{CDR21}.

For $k\in {{\mathbb Z}}$ let $Q=I_1\times I_2\times I_3$ be a dyadic box with $I_i$ dyadic intervals of length $2^{k\alpha_i}$, $\alpha_i \in {{\mathbb Z}}_+$. Denoting $\ell(Q)=2^k$, such a box partitions into $2^{\alpha_1+\alpha_2+\alpha_3}$ many children $Q_i$, each with $\ell(Q_i)=2^{k-1}$. For each $k$, these boxes partition ${{\mathbb R}}^3$. For any collection of such boxes $\mathcal{Q}$ we introduce a region in ${{\mathbb R}}^3\times (0,\infty)$
$$\Omega_\mathcal{Q} = \bigcup_{Q\in \mathcal{Q}} Q\times \big [\tfrac{\ell(Q)}{2},\ell(Q) \big ].$$

A collection of dyadic boxes $\mathcal{Q}$ is called {\em convex} if for any dyadic boxes $Q,Q',Q''$ we have that $Q \subseteq Q' \subseteq Q''$ and $Q,Q''\in \mathcal{Q}$ imply $Q'\in \mathcal{Q}$.
A finite collection of dyadic boxes $\mathcal{T}$ is called a {\em tree} if there exists $Q_\mathcal{T}\in \mathcal{T}$ which satisfies $Q\subseteq Q_\mathcal{T}$ for all $Q\in \mathcal{T}$.
By $\mathcal{L}(\mathcal{T})$ we denote the set of those boxes which do not belong to $\mathcal{T}$, but whose parent belongs to ${{\mathcal T}}$. Note that the set $\mathcal{L}(\mathcal{T})$ partitions $Q_\mathcal{T}$.

For a convex tree ${{\mathcal T}}$, $\lambda>1$, $r\in {{\mathbb R}}$, and a tuple of bounded real-valued measurable functions $F=(F_j)_{j\in J}$
we introduce the local forms
\[\Theta^{(1)}_{{{\mathcal T}},\lambda,r}(F) = \int_{\Omega_{{\mathcal T}}} \int_{{{\mathbb R}}^6} \Big( \prod_{j\in J} F_j(\Pi_j(x)) \Big) \h_{\lambda t^{\alpha_1},p_1}(x_1^0)\h_{\lambda t^{\alpha_1},p_1}(x_1^1) \]
\[\times \, \g_{\lambda t^{\alpha_2},p_2}(x_2^0)\g_{\lambda t^{\alpha_2},p_2}(x_2^1) \g_{t^{\alpha_3},p_3+rt^{\alpha_3}}(x_3^0)\g_{t^{\alpha_3},p_3+rt^{\alpha_3}}(x_3^1)
\, dx \, dp\, \frac{dt}{t}\]
\[\Theta^{(2)}_{{{\mathcal T}},\lambda,r}(F) = \int_{\Omega_{{\mathcal T}}} \int_{{{\mathbb R}}^6} \Big( \prod_{j\in J} F_j(\Pi_j(x)) \Big) \g_{\lambda t^{\alpha_1},p_1}(x_1^0)\g_{\lambda t^{\alpha_1},p_1}(x_1^1) \]
\[\times \, \h_{\lambda t^{\alpha_2},p_2}(x_2^0)\h_{\lambda t^{\alpha_2},p_2}(x_2^1) \g_{t^{\alpha_3},p_3+rt^{\alpha_3}}(x_3^0)\g_{t^{\alpha_3},p_3+rt^{\alpha_3}}(x_3^1)
\, dx\, dp \, \frac{dt}{t}\]
\[\Theta^{(3)}_{{{\mathcal T}},\lambda,r}(F) = \int_{\Omega_{{\mathcal T}}} \int_{{{\mathbb R}}^6} \Big( \prod_{j\in J} F_j(\Pi_j(x)) \Big) \g_{\lambda t^{\alpha_1},p_1}(x_1^0)\g_{\lambda t^{\alpha_1},p_1}(x_1^1) \]
\[\times \, \g_{\lambda t^{\alpha_2},p_2}(x_2^0)\g_{\lambda t^{\alpha_2},p_2}(x_2^1)\h_{t^{\alpha_3},p_3+rt^{\alpha_3}}(x_3^0)\h_{t^{\alpha_3},p_3+rt^{\alpha_3}}(x_3^1)
\, dx \, dp \frac{dt}{t}\]
We also define
\begin{equation}
\mathfrak{C}_{\mathcal T}(F)
=
\sup_{(p,t)\in\overline{\Omega_{\mathcal T}}}
\int_{{{\mathbb R}}^6}
\Big(\prod_{j\in J}|F_j(\Pi_j(x))|\Big)
\prod_{i=1}^3\prod_{b=0}^1
\theta_{t^{\alpha_i},p_i}(x_i^b)\,dx.
\label{E:localized-cube-functional}
\end{equation}

\begin{lemma}[The telescoping identity] \label{lem:telescoping_main}
Let $\lambda>1$ and $r\in{{\mathbb R}}$. There exists a constant
$
C_{\lambda,r}=O(\lambda^{37}(1+|r|)^{21})
$
such that for every convex tree $\mathcal T$, there is a multilinear
form $\mathcal R_{\mathcal T,\lambda,r}$ satisfying
\begin{equation}
\sum_{i=1}^3\alpha_i\Theta^{(i)}_{\mathcal T,\lambda,r}(F)
=\mathcal R_{\mathcal T,\lambda,r}(F)
\label{E:general-telescoping-identity}
\end{equation}
and
\begin{equation}
|\mathcal R_{\mathcal T,\lambda,r}(F)|
\le C_{\lambda,r}|Q_{\mathcal T}|\,
\mathfrak{C}_{\mathcal T}(F).
\label{E:general-telescoping-remainder}
\end{equation}
\end{lemma}

The proof of this lemma closely follows \cite[Lemma 4.3]{CDR21}. The argument relies on the heat equation satisfied by the Gaussian kernels, which converts a derivative in the scale into a spatial derivative of $\h$.
Namely, differentiating the product $\g_{t^{\alpha_i},p_i}(x_i^0)\g_{t^{\alpha_i},p_i}(x_i^1)$ in $t$ and integrating by parts in $p_i$ expresses the integral of $-t\partial_t$ of this product over an interval in $p_i$ as a constant multiple of the integral of $\h_{t^{\alpha_i},p_i}(x_i^0)\h_{t^{\alpha_i},p_i}(x_i^1)$ over the same interval, plus boundary terms at its two endpoints. Applying the fundamental theorem of calculus in $t$ over $[\ell(Q)/2,\ell(Q)]$ for a single box $Q\in{{\mathcal T}}$, together with the product rule, one thus writes the left-hand side of \eqref{E:general-telescoping-identity} as the difference of the forms with the Gaussian kernels at the scales $\ell(Q)/2$ and $\ell(Q)$, plus boundary terms on the faces of $Q$.

Summing over $Q\in{{\mathcal T}}$, the convexity of ${{\mathcal T}}$ makes the Gaussian terms at consecutive scales telescope, leaving only the Gaussian forms at the root $Q_{{\mathcal T}}$ and at the leaves $\mathcal{L}({{\mathcal T}})$. Since $\mathcal{L}({{\mathcal T}})$ partitions $Q_{{\mathcal T}}$, these are bounded by $|Q_{{\mathcal T}}|\mathfrak{C}_{{\mathcal T}}(F)$. To control the boundary terms, after cancellation of the contributions on faces shared by two adjacent boxes of the same scale, only faces in the boundary of the union of the boxes of a fixed scale survive, and a counting argument based on the convexity of ${{\mathcal T}}$ bounds their total contribution by $|Q_{{\mathcal T}}|\mathfrak{C}_{{\mathcal T}}(F)$. We omit the details of this proof.

For a collection $\mathcal{Q}$ of dyadic boxes and an integer $n$ we denote
\begin{equation}\label{def:mq}
\mathcal{M}_{\mathcal{Q},n}(f)=
\sup_{Q\in\mathcal{Q}} \Big(\;|f|^{n} * \Theta_{\ell(Q)} (c(Q)) \Big )^{1/n} ,
\end{equation}
where $c(Q)$ denotes the center of $Q$.
Note that
\[ \mathcal{M}_{\mathcal{Q},n}(f) \asymp
\sup_{\substack{(x,t)\in \Omega_{\mathcal{Q}}}}\Big(\;|f|^{n} * \Theta_{t}(x) \Big )^{1/n}. \]

Combining Lemma \ref{lem:telescoping_main} with the Brascamp-Lieb estimates from Lemmas \ref{lem:BL4} and \ref{lem:BL2}, we immediately obtain the following corollary.
\begin{corollary}\label{lem:telescoping}
Let $\lambda>1$ and $r\in{{\mathbb R}}$. There exists a constant
$
C_{\lambda,r}=O(\lambda^{37}(1+|r|)^{21})
$
such that for every convex tree $\mathcal T$, there is a multilinear
form $\mathcal R_{\mathcal T,\lambda,r}$ satisfying \eqref{E:general-telescoping-identity}
and such that for an arbitrary tuple $F$,
\begin{equation}
|\mathcal R_{\mathcal T,\lambda,r}(F)|
\le C_{\lambda,r}|Q_{\mathcal T}|
\prod_{j\in J}\mathcal M_{\mathcal T,4}(F_j).
\label{E:general-telescoping-L4}
\end{equation}
If
$
F_j=1$ for $j\notin\{(0,0,0),(0,0,1)\}$,
then one also has
\begin{equation}
|\mathcal R_{\mathcal T,\lambda,r}(F)|
\le C_{\lambda,r}|Q_{\mathcal T}|\,
\mathcal M_{\mathcal T,2}(F_{(0,0,0)})
\mathcal M_{\mathcal T,2}(F_{(0,0,1)}).
\label{E:general-telescoping-L2}
\end{equation}
\end{corollary}

For $i\in\{1,2,3\}$ and $b\in\{0,1\}$, define the
operator $\mathsf C_i^b$ on tuples
$F=(F_j)_{j\in J}$ by
\[
(\mathsf C_i^bF)_{(j_1,j_2,j_3)}
=
F_{(j_1,\ldots,j_{i-1},b,j_{i+1},\ldots,j_3)}.
\]

\begin{lemma}[Cauchy--Schwarz]
\label{lem:cs}
Let $F=(F_j)_{j\in J}$ be a tuple of real-valued functions. Then,
for every $i\in\{1,2,3\}$, every convex tree $\mathcal T$,
$\lambda>1$, and $r\in{{\mathbb R}}$,
\begin{equation}
\big|
\Theta^{(i)}_{\mathcal T,\lambda,r}(F)
\big|^2
\leq
\Theta^{(i)}_{\mathcal T,\lambda,r}(\mathsf C_i^0F)
\,
\Theta^{(i)}_{\mathcal T,\lambda,r}(\mathsf C_i^1F).
\label{E:cube-CS}
\end{equation}
In particular, $
\Theta^{(i)}_{\mathcal T,\lambda,r}(\mathsf C_i^bF)\ge0$ for
$b\in\{0,1\}$.
\end{lemma}

The proof is by Fubini and a single application of the Cauchy-Schwarz inequality, as in \cite{CDR21}. We omit the details.

Now we are ready to prove the main tree estimate.
\begin{prop}[Full cube tree estimate] \label{prop:fullcubetree}
Let $\lambda>1$, $r\in {{\mathbb R}}$, and $i\in \{1,2,3\}$. There exists a constant $C_{\lambda,r} =O(\lambda^{37}(1+|r|)^{21})$ such that
for any convex tree $\mathcal{T}$ and any tuple $F=(F_j)_{j\in J}$ of bounded real-valued measurable functions,
\[|\Theta^{(i)}_{{{\mathcal T}},\lambda,r}(F)| \le C_{\lambda,r} |Q_{{\mathcal T}}| \prod_{j\in J}\mathcal{M}_{{{\mathcal T}},4}(F_j).\]
\end{prop}
\begin{proof}
If $\mathcal M_{\mathcal T,4}(F_j)=0$ for some $j$, then
$F_j=0$ almost everywhere and the conclusion is immediate.
Thus we may assume that all these quantities are positive.

By anisotropic scaling and multilinearity, we may assume
\[
|Q_{\mathcal T}|=1, \quad \mathcal M_{\mathcal T,4}(F_j)=1
\]
for $j\in J$.
We write $C$ for any
quantity bounded by
$O(\lambda^{37}(1+|r|)^{21})$, enlarging it so that $C\ge1$. Throughout the proof we suppress the subscripts $\mathcal T,\lambda,r$.

Fix $i\in\{1,2,3\}$. Lemma \ref{lem:cs} gives
\begin{equation}
|\Theta^{(i)}(F)|^2
\leq
\Theta^{(i)}(\mathsf C_i^0F)
\Theta^{(i)}(\mathsf C_i^1F).
\label{E:full-cube-first-CS}
\end{equation}
It is therefore enough to show that for $a\in\{0,1\}$ and $F^a = \mathsf C_i^aF$,
\[
\Theta^{(i)}(F^a) \le C .
\]
The telescoping identity from Corollary \ref{lem:telescoping} gives
\begin{equation}
\alpha_i\Theta^{(i)}(F^a)
=
\mathcal R(F^a)
-
\sum_{\substack{1\le q\le3\\q\neq i}}
\alpha_q\Theta^{(q)}(F^a) \leq
C+
\sum_{\substack{1\le q\le3\\q\neq i}}
\alpha_q|\Theta^{(q)}(F^a)|.
\label{E:full-cube-first-reduction}
\end{equation}
where we used $|\mathcal{R}(F^a)|\le C$.
We next estimate the two forms on the right-hand side.

Fix $q\neq i$, and let $p$ be the unique element of
$\{1,2,3\}\setminus\{i,q\}$.
By Lemma \ref{lem:cs},
\begin{equation}
|\Theta^{(q)}(F^a)|^2
\leq
\Theta^{(q)}(\mathsf C_q^0F^a)
\Theta^{(q)}(\mathsf C_q^1F^a).
\label{E:full-cube-second-CS}
\end{equation}
For $b\in\{0,1\}$, set
\[
F^{a,b}
=
\mathsf C_q^bF^a
=
\mathsf C_q^b\mathsf C_i^aF.
\]
Since $\mathsf C_i^a$ and $\mathsf C_q^b$ act in distinct coordinates and
commute, $F^{a,b}$ belongs to the range of both
$\mathsf C_i^a$ and $\mathsf C_q^b$. The nonnegativity conclusion
of Lemma~\ref{lem:cs} therefore gives
\begin{equation}
\Theta^{(i)}(F^{a,b})\geq 0,
\qquad
\Theta^{(q)}(F^{a,b})\geq 0.
\label{E:full-cube-two-coordinate-positivity}
\end{equation}
Using Corollary \ref{lem:telescoping} again and discarding the non-negative
$i$th-coordinate term,
\begin{equation}
\alpha_q\Theta^{(q)}(F^{a,b})
\leq
|\mathcal R(F^{a,b})|
+
\alpha_p|\Theta^{(p)}(F^{a,b})| \le
C+\alpha_p|\Theta^{(p)}(F^{a,b})|.
\label{E:full-cube-second-reduction}
\end{equation}

We estimate the remaining form by one final application of
Cauchy--Schwarz as
\begin{equation}
|\Theta^{(p)}(F^{a,b})|^2
\leq
\Theta^{(p)}(\mathsf C_p^0F^{a,b})
\Theta^{(p)}(\mathsf C_p^1F^{a,b}).
\label{E:full-cube-third-CS}
\end{equation}
Now let $
F^{a,b,c}
=
\mathsf C_p^cF^{a,b}$ and by Lemma \ref{lem:cs} we have
\[
\Theta^{(\ell)}(F^{a,b,c})\ge0
\]
for $\ell=1,2,3$.
Corollary \ref{lem:telescoping} gives
\[
\alpha_p\Theta^{(p)}(F^{a,b,c})
\leq
\sum_{\ell=1}^3
\alpha_\ell\Theta^{(\ell)}(F^{a,b,c})
=
\mathcal R(F^{a,b,c}) \le C.
\]
This gives the desired estimate, after tracing the bounds back through the three applications of the Cauchy–Schwarz inequality.
\end{proof}

\begin{prop}[Edge tree estimate]
\label{prop:f3bd}Let $\lambda>1$, $r\in {{\mathbb R}}$. Assume that $(F_j)_{j\in J}$ is a tuple satisfying $F_{(0,0,0)}=F_{(0,0,1)}$ and $F_j=1$ for all other $j\in J$.
Then
\begin{equation}
\label{formszero}\Theta^{(1)}_{{{\mathcal T}},\lambda,r}(F)=\Theta^{(2)}_{{{\mathcal T}},\lambda,r}(F) =0,
\end{equation}
and
there exists a constant $C_{\lambda,r} =O(\lambda^{37}(1+|r|)^{21})$ such that
for any convex tree $\mathcal{T}$,
\begin{equation}
|\Theta^{(3)}_{{{\mathcal T}},\lambda,r}(F)|\le C_{\lambda,r} |Q_{{\mathcal T}}| (\mathcal{M}_{{{\mathcal T}},2}(F_{(0,0,0)}))^2
\end{equation}
\end{prop}

\begin{proof} Let
$ f=F_{(0,0,0)}=F_{(0,0,1)}$.
Since the product of the input functions is independent of
$x_1^1$ and $x_2^1$, and since $\int_{{{\mathbb R}}}\h=0$, integration in
$x_1^1$ and $x_2^1$ gives \[\Theta^{(1)}_{{{\mathcal T}},\lambda,r}(F)=\Theta^{(2)}_{{{\mathcal T}},\lambda,r}(F) =0.\]
The telescoping identity
\eqref{E:general-telescoping-identity} therefore reduces to
\[
\alpha_3\Theta^{(3)}_{\mathcal T,\lambda,r}(F)
=\mathcal R_{\mathcal T,\lambda,r}(F).
\]
By the assumption on the tuple, the improved
remainder estimate \eqref{E:general-telescoping-L2} gives
\[
|\mathcal R_{\mathcal T,\lambda,r}(F)|
\le
C_{\lambda,r}|Q_{\mathcal T}|
\big(\mathcal M_{\mathcal T,2}(f)\big)^2.
\]
Dividing by $\alpha_3$ proves the desired estimate, after absorbing
$\alpha_3^{-1}$ into $C_{\lambda,r}$.
\end{proof}

\subsection{\texorpdfstring{From $\Theta_{{{\mathcal T}},\lambda,r}$ to $\Lambda_{{{\mathcal T}},u}$}{Towards the localized forms}}
In this section we define $\Lambda_{{{\mathcal T}},u}$ which is a localized variant of $\Lambda_u$, and deduce bounds for $\Lambda_{{{\mathcal T}},u}$ from Propositions \ref{prop:fullcubetree} and \ref{prop:f3bd}.

For a finite collection of anisotropic dyadic boxes $\mathcal{Q}$ and a tuple of real-valued Schwartz functions $f=(f_0,f_1,f_2,f_3)$ we define
\[\Lambda_{\mathcal{Q},u}(f) = \int_{\Omega_{\mathcal{Q}}} c_3(t,u) \int_{{{\mathbb R}}^6}
f_0(x_1,x_2,x_3)
f_1(y_1,x_2,x_3)f_2(x_1,y_2,x_3)f_3(x_1,x_2,y_3)\]
\[\times ({\varphi}_{1,u_1})_{t^{\alpha_1},p_1}(x_1)\g_{t^{\alpha_1},p_1}(y_1)
({\varphi}_{1,u_2})_{t^{\alpha_2},p_2}(x_2)\g_{t^{\alpha_2},p_2}(y_2)
\h_{t^{\alpha_3},p_3}(x_3)(\h*{\psi}_{1,u_3})_{t^{\alpha_3},p_3}(y_3) \, dx \, dy\, dp\, \frac{dt}{t},
\]
where ${\varphi}_{1,u_1},\,{\varphi}_{1,u_2},\,{\psi}_{1,u_3}$ are as in the definition of $\Lambda_u$.

For the rest of this section we restrict our attention to the case when $\mathcal{Q}$ is a finite convex tree $\mathcal{T}$. We define two further local forms on real-valued Schwartz functions $f_0,f_1,f_2$ and $f_3$, respectively, by
\[\Lambda_{\mathcal{T},u,1}(f_0,f_1,f_2) = \int_{\Omega_{\mathcal{T}}} \int_{{{\mathbb R}}^5}
\Big( \int_{{{\mathbb R}}}f_0(x_1,x_2,x_3)
f_1(y_1,x_2,x_3)f_2(x_1,y_2,x_3) \h_{t^{\alpha_3},p_3}(x_3) dx_3\Big)^2 d\mu \, dp\, \frac{dt}{t}\]
and
\[\Lambda_{\mathcal{T},u,2}(f_3)= \int_{\Omega_{\mathcal{T}}} \int_{{{\mathbb R}}^5}
\Big( \int_{{{\mathbb R}}}f_3(x_1,x_2,y_3) \h_{t^{\alpha_3}, p_3+r t^{\alpha_3}}(y_3) dy_3\Big)^2 d\mu \, dp\, \frac{dt}{t} ,\]
where
\[d\mu = |({\varphi}_{1,u_1})_{t^{\alpha_1},p_1}(x_1)|\g_{t^{\alpha_1},p_1}(y_1)|({\varphi}_{1,u_2})_{t^{\alpha_2},p_2}(x_2)|\g_{t^{\alpha_2},p_2}(y_2)|\psi_{1,u_3}(r)| dx_1dx_2 dy_1dy_2 dr .\]

\begin{lemma}
\label{lem:lambda12-cs}
For any convex tree $\mathcal{T}$ and $u\in {{\mathbb R}}^3$,
\begin{equation}
\label{E:first-local-CS}|\Lambda_{\mathcal{T},u}(f)|^2 \le \Lambda_{\mathcal{T},u,1}(f_0,f_1,f_2)\,\Lambda_{\mathcal{T},u,2}(f_3),
\end{equation}

\end{lemma}

\begin{proof}
Using
\[(\h*{\psi}_{1,u_3})_{t^{\alpha_3},p_3}(y_3) = \int_{{{\mathbb R}}} \h_{t^{\alpha_3}, p_3+r t^{\alpha_3}}(y_3)\psi_{1,u_3}(r)dr\]
we can write
\[\Lambda_{\mathcal{T},u}(f) = \int_{\Omega_{\mathcal{T}}} c_3(t,u) \int_{{{\mathbb R}}^5}
\Big( \int_{{{\mathbb R}}}f_0(x_1,x_2,x_3)
f_1(y_1,x_2,x_3)f_2(x_1,y_2,x_3)\h_{t^{\alpha_3},p_3}(x_3) dx_3 \Big) \]
\[
\times \Big( \int_{{{\mathbb R}}}f_3(x_1,x_2,y_3) \h_{t^{\alpha_3}, p_3+r t^{\alpha_3}}(y_3) dy_3 \Big)
({\varphi}_{1,u_1})_{t^{\alpha_1},p_1}(x_1)\g_{t^{\alpha_1},p_1}(y_1)\]
\[
\times ({\varphi}_{1,u_2})_{t^{\alpha_2},p_2}(x_2)\g_{t^{\alpha_2},p_2}(y_2)\psi_{1,u_3}(r)
\, dx_1dx_2 dy_1dy_2 dr dp \frac{dt}{t}
\]
The triangle inequality, $|c_3(t,u)|\le 1$, and the Cauchy-Schwarz inequality
give the desired estimate.
\end{proof}

Each of the forms on the right-hand side satisfies an estimate in terms of the corresponding maximal sizes.
\begin{lemma}
\label{lem:lambda1}
For any convex tree $\mathcal{T}$ and $u\in {{\mathbb R}}^3$,
\[ \Lambda_{\mathcal{T},u,1}(f_0,f_1,f_2) \leq C\langle u\rangle^{50} |Q_\mathcal{T}| \prod_{j=0}^2 \mathcal{M}_{\mathcal{T},4}(f_j)^2,
\]
with a constant $C$ independent of ${{\mathcal T}}$, $u$, and $f_0,f_1,f_2$.
\end{lemma}

\begin{proof}
We estimate the Schwartz functions by superpositions of Gaussians.
For $z\in {{\mathbb R}}^4$, from the Schwartz decay of $\varphi$ we obtain
\[|{\varphi}_{1,u_1}(z_1)|\otimes \g(z_2) \otimes |{\varphi}_{1,u_2}(z_3)| \otimes \g(z_4) \le C \langle u \rangle^{50} (1+|z|)^{-50}\]
On the other hand,
\[
\int_1^\infty\prod_{m=1}^4\g_\lambda(z_m)\lambda^{-47}\,d\lambda
=
\int_1^\infty
e^{-\pi|z|^2/\lambda^2}\lambda^{-51}\,d\lambda
\ge c
(1+|z|)^{-50}
\]
for some constant $c>0$.
Therefore,
\begin{equation} |{\varphi}_{1,u_1}|\otimes \g \otimes |{\varphi}_{1,u_2}| \otimes \g \le C \langle u \rangle^{50} \int_{1}^\infty \g_\lambda\otimes \g_\lambda \otimes \g_\lambda\otimes \g_\lambda \lambda^{-47}d\lambda \label{E:four-variable-Gaussian-superposition}
\end{equation}
We also have
\begin{equation}
\label{psibd}
\int_{{{\mathbb R}}}|\psi_{1,u_3}(r)|\,dr=\|\psi\|_1.
\end{equation}

Using \eqref{E:four-variable-Gaussian-superposition} and
\eqref{psibd}, and then expanding the square in
$\Lambda_{\mathcal T,u,1}$, we obtain
\begin{equation}
\Lambda_{\mathcal T,u,1}(f_0,f_1,f_2)
\le C
\langle u\rangle^{50}
\int_1^\infty
\Theta^{(3)}_{\mathcal T,\lambda,0}(F^{[1]})
\lambda^{-47}\,d\lambda,
\label{lambdabytheta_1}
\end{equation}
where $F^{[1]}$ is independent of its third index and is given,
for $j_3\in\{0,1\}$, by
\[
F^{[1]}_{(0,0,j_3)}=f_0,
\qquad
F^{[1]}_{(1,0,j_3)}=f_1,
\]
\[
F^{[1]}_{(0,1,j_3)}=f_2,
\qquad
F^{[1]}_{(1,1,j_3)}=1.
\]
Applying
Proposition \ref{prop:fullcubetree} to the form on the right-hand side of \eqref{lambdabytheta_1} gives
\[
\Lambda_{\mathcal T,u,1}(f_0,f_1,f_2)
\le C
\langle u\rangle^{50}|Q_{\mathcal T}|
\prod_{j=0}^2\mathcal M_{\mathcal T,4}(f_j)^2
\int_1^\infty\lambda^{37}\lambda^{-47}\,d\lambda \]
Integrating in $\lambda$ gives the desired estimate.
\end{proof}

\begin{lemma}
\label{lem:lambda2}
For any convex tree $\mathcal{T}$ and $u\in {{\mathbb R}}^3$,
\begin{equation}
\label{Lambda2_loc}
\Lambda_{\mathcal{T},u,2}(f_3) \le C \min( \langle u\rangle^{100}|Q_{\mathcal T}|
\mathcal M_{\mathcal T,2}(f_3)^2, \|f_3\|^2_2)
\end{equation}
with a constant $C$ independent of ${{\mathcal T}}$, $u$, and $f_0,f_1,f_2$.
\end{lemma}

\begin{proof} First we prove
\[ \Lambda_{\mathcal{T},u,2}(f_3) \le C \langle u\rangle^{100} |Q_{\mathcal T}|
\mathcal M_{\mathcal T,2}(f_3)^2.\]
Using \eqref{E:four-variable-Gaussian-superposition} and the estimate
\begin{equation}
|\psi_{1,u_3}(r)|
\le C
(1+|u_3|)^{50}(1+|r|)^{-50}
\le C
\langle u\rangle^{50}(1+|r|)^{-50},
\end{equation} and then expanding the square, we obtain
\begin{equation}
\Lambda_{\mathcal T,u,2}(f_3)
\le C
\langle u\rangle^{100}
\int_1^\infty\int_{{{\mathbb R}}}
\Theta^{(3)}_{\mathcal T,\lambda,r}(F^{[2]})(1+|r|)^{-50}\lambda^{-47}\,dr\,d\lambda,
\label{lambdabytheta_2}
\end{equation}
where
\[
F^{[2]}_{(0,0,0)}=F^{[2]}_{(0,0,1)}=f_3
\]
and $F^{[2]}_j=1$ for all other $j\in J$.

Proposition \ref{prop:f3bd} applied to the form on the right-hand side gives
\[
\Lambda_{\mathcal T,u,2}(f_3)
\le C
\langle u\rangle^{100}|Q_{\mathcal T}|
\mathcal M_{\mathcal T,2}(f_3)^2
\int_1^\infty\lambda^{37}\lambda^{-47}\,d\lambda
\int_{{{\mathbb R}}}(1+|r|)^{21}(1+|r|)^{-50}\,dr\]
The desired estimate now follows by integrating in $\lambda$ and $r$.

Next we prove
\[\Lambda_{\mathcal T,u,2}(f_3)
\le C \|f_3\|^2_2 .\]
Enlarging the $(p,t)$-domain from $\Omega_{\mathcal T}$ to
${{\mathbb R}}^3\times(0,\infty)$, we obtain
\[\Lambda_{\mathcal{T},u,2}(f_3)\le \int_{0}^\infty \int_{{{\mathbb R}}^3} \int_{{{\mathbb R}}^5}
\Big( \int_{{{\mathbb R}}}f_3(x_1,x_2,y_3) \h_{t^{\alpha_3}, p_3+r t^{\alpha_3}}(y_3) dy_3\Big)^2 d\mu \, dp\, \frac{dt}{t}\]
Then we integrate in $y_1,y_2$ and use $\int_{{{\mathbb R}}}\g=1$.
Next we integrate in $p_1,p_2$ and use
\[
\int_{{{\mathbb R}}}
|({\varphi}_{1,u_i})_{t^{\alpha_i},p_i}(x_i)|\,dp_i
=
\|\varphi\|_1,
\]
for $i=1,2$.
Finally, the change of variables
$q=p_3+rt^{\alpha_3}$ and \eqref{psibd} give
\begin{equation}
\Lambda_{\mathcal{T},u,2}(f_3)\le C\int_{{{\mathbb R}}^2}\int_0^\infty\int_{{{\mathbb R}}}
|H_t f_3(x_1,x_2,q) |^2
\,dq\,\frac{dt}{t}\,dx_1dx_2,
\label{sqftc}
\end{equation}
where
\[H_t f_3(x_1,x_2,q) = \int_{{{\mathbb R}}}
f_3(x_1,x_2,y_3)\h_{t^{\alpha_3},q}(y_3)\,dy_3 \]
For fixed $x_1,x_2$, Plancherel's theorem in the third variable gives
\[
\int_0^\infty\int_{{{\mathbb R}}}
| H_t f_3(x_1,x_2,q)|^2
\,dq\,\frac{dt}{t}
=
\int_{{{\mathbb R}}}
|\widehat{f_3(x_1,x_2,\cdot)}(\xi_3)|^2
\Big (
\int_0^\infty
|\widehat{\h}(t^{\alpha_3}\xi_3)|^2
\frac{dt}{t}
\Big )d\xi_3.
\]
The $t$-integral equals a finite constant independent of $\xi_3\ne 0$ because
$\widehat{\h}$ vanishes at the origin and is rapidly decreasing, so the last display is
\[ C \int_{{{\mathbb R}}}
|\widehat{f_3(x_1,x_2,\cdot)}(\xi_3)|^2 d\xi_3 = C \int_{{{\mathbb R}}}
|{f_3(x_1,x_2,y_3)}|^2 dy_3 \]
Here we used Plancherel once more. We obtain a bound for \eqref{sqftc} by $C\|f_3\|_2^2$.
\end{proof}

\subsection{\texorpdfstring{From $\Lambda_{{{\mathcal T}},u}$ to $\Lambda_u$, stopping time}{Stopping time}}
\label{sec:lambdaT_to_Lambda}
We first state a general stopping-time lemma.
\begin{lemma}
\label{lem:abstract-stopping-time}
Let $d\geq 2$, and let
$
1\leq s_j<q_j<\infty,
0\leq j\leq d-1$,
satisfy
\begin{equation}
\sum_{j=0}^{d-1}\frac1{q_j}=1.
\label{E:stopping-exponent-sum}
\end{equation}
There exists a constant $C>0$ such that the following holds.

For $0\le j \le d-1$ let $g_j$ be non-zero Schwartz functions on ${{\mathbb R}}^3$, and
let $\mathcal Q_0$ be a finite convex collection of anisotropic
dyadic boxes.
For $\mathbf n=(n_j)_{j=0}^{d-1}\in{{\mathbb Z}}^d$, let
\begin{equation}
\mathcal P_{\mathbf n}
=
\Big\{
Q\in\mathcal Q_0:
2^{n_j-1}<\sup_{\substack{Q'\in\mathcal Q_0\\Q\subseteq Q'}}
\mathcal M_{\{Q'\},s_j}(g_j)\leq 2^{n_j}
\text{ for }0\le j \le d-1
\Big\}.
\label{E:stopping-level-collection}
\end{equation}
Let $\mathcal P_{\mathbf n}^{\max}$ denote the collection of maximal
elements of $\mathcal P_{\mathbf n}$. For
$Q\in\mathcal P_{\mathbf n}^{\max}$, define
\begin{equation}
\mathcal T_Q
=
\left\{
Q'\in\mathcal P_{\mathbf n}:Q'\subseteq Q
\right\}.
\label{E:stopping-tree-definition}
\end{equation}
Then the collections $\mathcal T_Q$ form a partition of
$\mathcal Q_0$ into convex trees and
\begin{equation}
\sum_{\mathbf n\in{{\mathbb Z}}^d}
\sum_{Q\in\mathcal P_{\mathbf n}^{\max}}
|Q|
\prod_{j=0}^{d-1}
\mathcal M_{\mathcal T_Q,s_j}(g_j)
\le C
\prod_{j=0}^{d-1}\|g_j\|_{q_j}.
\label{E:abstract-stopping-product-sum}
\end{equation}
\end{lemma}

The proof of this lemma is standard and follows along the lines of the arguments in \cite[Section 4.4]{CDR21} and \cite{Kov12, Dur15}.
We omit the details.

We now deduce the estimate for $\Lambda_{u}$ and conclude the proof of Proposition \ref{prop:twisted}.
By splitting the functions into their real and imaginary parts and
using multilinearity of $\Lambda_{u}$, we may assume all input functions are real valued. Exhausting ${{\mathbb R}}^3\times(0,\infty)$ by regions associated with finite convex collections and
a limiting argument, it suffices to prove an estimate for $\Lambda_{\mathcal{Q}_0,u}$ for any finite
convex collection $\mathcal Q_0$ of anisotropic dyadic boxes, with bounds independent of the collection. We now fix one such $\mathcal Q_0$.

Let $p_0,p_1,p_2$ be as in Proposition \ref{prop:twisted}. We first consider the case $2<p_3<\infty$.
Lemma
\ref{lem:abstract-stopping-time} partitions $\mathcal Q_0$ into
convex trees $\mathcal T_Q$. Since the regions
$\Omega_{\mathcal T_Q}$ are pairwise disjoint up to null sets and
have union $\Omega_{\mathcal Q_0}$,
\begin{equation}
\Lambda_{\mathcal Q_0,u}(f)= \sum_{\mathbf n}
\sum_{Q\in\mathcal P_{\mathbf n}^{\max}}
\Lambda_{\mathcal T_Q,u}(f).
\label{Lambda_treedec}
\end{equation}
For each tree $\mathcal T_Q$, Lemma \ref{lem:lambda12-cs} gives
\begin{equation}
\label{stop_cs}
|\Lambda_{\mathcal T_Q,u}(f)|^2
\leq
\Lambda_{\mathcal T_Q,u,1}(f_0,f_1,f_2)
\Lambda_{\mathcal T_Q,u,2}(f_3)
\end{equation}
Lemmas
\ref{lem:lambda1} and \ref{lem:lambda2}, applied to \eqref{stop_cs} give for each tree $\mathcal T_Q$
\[ |\Lambda_{\mathcal T_Q,u}(f)|
\le C
\langle u\rangle^{75}|Q|
\Big(
\prod_{j=0}^2
\mathcal M_{\mathcal T_Q,4}(f_j)
\Big)
\mathcal M_{\mathcal T_Q,2}(f_3).
\]
Using \eqref{Lambda_treedec} and the triangle inequality, we obtain
\[
|\Lambda_{\mathcal Q_0,u}(f)|
\le C
\langle u\rangle^{75}
\sum_{\mathbf n}
\sum_{Q\in\mathcal P_{\mathbf n}^{\max}}
|Q|
\Big(
\prod_{j=0}^2
\mathcal M_{\mathcal T_Q,4}(f_j)
\Big)
\mathcal M_{\mathcal T_Q,2}(f_3)
\le C
\langle u\rangle^{75}\prod_{j=0}^3\|f_j\|_{p_j}.\]
The last bound follows by Lemma \ref{lem:abstract-stopping-time} applied with $d=4$, $s_0=s_1=s_2=4,\;
s_3=2$,
$q_j=p_j$, and $g_j=f_j$.

We next consider the endpoint $p_3=2$. Using Lemma \ref{lem:abstract-stopping-time} we again partition as in \eqref{Lambda_treedec}.
By \eqref{stop_cs} and Cauchy-Schwarz over the family of trees ${{\mathcal T}}_Q$,
\[|\Lambda_{\mathcal Q_0,u}(f)| \le \Big( \sum_{{{\mathcal T}}_Q} \Lambda_{\mathcal T_Q,u,1}(f_0,f_1,f_2)
\Big )^{1/2} \Big( \sum_{{{\mathcal T}}_Q}\Lambda_{\mathcal T_Q,u,2}(f_3) \Big )^{1/2} \]
Lemma \ref{lem:lambda1} gives
\begin{equation}
\label{lambda1_stop}
\sum_{{{\mathcal T}}_Q} \Lambda_{\mathcal T_Q,u,1}(f_0,f_1,f_2) \le C\langle u \rangle^{50} \sum_{\mathbf n}
\sum_{Q\in\mathcal P_{\mathbf n}^{\max}}
|Q|
\prod_{j=0}^2
\mathcal M_{\mathcal T_Q,4}(f_j)^2
\end{equation}
Denoting $g_j = f_j^2$, we have
\[\mathcal M_{\mathcal T_Q,4}(f_j)^2 = \mathcal M_{\mathcal T_Q,2}(g_j)\]
We apply Lemma \ref{lem:abstract-stopping-time} with $d=3$, $q_j=p_j/2, \, s_j =2$. Since $p_3=2$ and the exponents satisfy H\"older scaling,
$
\sum_{j=0}^21/{q_j}
=1.$
Moreover, $p_j>4$ implies $q_j=p_j/2>2=s_j$. Lemma~\ref{lem:abstract-stopping-time} then bounds \eqref{lambda1_stop} by
\[C\langle u \rangle^{50} \prod_{j=0}^2 \|g_j\|_{q_j} = C\langle u \rangle^{50} \prod_{j=0}^2 \|f_j\|^2_{p_j} \]
Now, Lemma \ref{lem:lambda2} applied to the collection $\mathcal{Q}_0$ (the proof of the $L^2$ bound in Lemma \ref{lem:lambda2} does not use the tree structure) gives
\[\sum_{{{\mathcal T}}_Q}\Lambda_{\mathcal T_Q,u,2}(f_3) = \Lambda_{\mathcal{Q}_0,u,2}(f_3) \le C \|f_3\|^2_2 \]
Combining these estimates finishes the proof for $\Lambda_{\mathcal{Q}_0,u}$, and hence the proof of Proposition \ref{prop:twisted}.

\end{document}